\documentclass{amsart}
\usepackage[utf8]{inputenc}
\usepackage{fullpage}
\usepackage{parskip}
\usepackage{amsmath}
\usepackage{amsfonts}
\usepackage{amssymb}
\usepackage{mathrsfs}
\usepackage{tikz-cd}
\usepackage{graphicx}
\usepackage{comment}
\usepackage{todonotes}
\usepackage{hyperref}

\newtheorem{theorem}{Theorem}[section]
\newtheorem{prop}[theorem]{Proposition}
\newtheorem{definition}[theorem]{Definition}

\newtheorem{corollary}{Corollary}
\newtheorem{lemma}[theorem]{Lemma}
\newtheorem{remark}[theorem]{Remark}

\newcommand{\N}{\mathbb{N}}
\newcommand{\Z}{\mathbb{Z}}

\newcommand{\G}{\mathbb{G}}
\newcommand{\CD}{\mathit{CD}}
\newcommand{\CDP}{\mathit{CDP}}
\newcommand{\Id}{\mathrm{Id}}

\title{Obstruction Complexes in Grid Homology}
\author{Yan Tao}
\date{}

\begin{document}


\begin{abstract}
Recently, Manolescu-Sarkar constructed a stable homotopy type for link Floer homology, which uses grid homology and accounts for all domains that do not pass through a specific square. In doing so, they produced an obstruction chain complex of the grid diagram with that square removed. We define the obstruction chain complex of the full grid, without the square removed, and compute its homology. Though this homology is too complicated to immediately extend the Manolescu-Sarkar construction, we give results about the existence of sign assignments in grid homology.
\end{abstract}

\maketitle

\section{Introduction}

Link Floer homology, developed by \cite{OS3}, \cite{RA}, and \cite{OS4}, is an invariant of oriented links in three-manifolds which comes from Heegaard Floer homology, from \cite{OS1} and \cite{OS2}. \cite{MOS}, \cite{MOST}, and \cite{OSS} gave a combinatorial description of the link Floer chain complex for a link in $S^3$ using grid diagrams, known as grid homology. A toroidal grid diagram is a $n \times n$ grid of squares, with the left and right edges identified and the top and bottom edges identified, together with markings $X$ and $O$, such that each row and column contains exactly one $X$ and one $O$. Given a grid diagram $\G$, drawing vertical segments from the $X$ to the $O$ in each column and horizontal segments---going under the vertical segments whenever they cross---from the $O$ to the $X$ in each row gives the diagram of an oriented link $L$; we say that $\G$ is a grid diagram for $L$. Figure $\ref{fig1}$ shows a $5 \times 5$ grid diagram for the trefoil. The grid chain complex is generated by unordered $n$-tuples of intersection points between the horizontal and vertical circles---Figure $\ref{fig1}$ shows an example of such a generator.

\begin{figure}\begin{tikzpicture}

\draw (0, 0) grid (5, 5);
\draw (0.5, 4.5) -- (3.5, 4.5);
\draw (0.5, 1.5) -- (0.5, 4.5);
\draw (1.5, 0.5) -- (1.5, 2.5);
\draw (1.5, 0.5) -- (4.5, 0.5);
\draw (4.5, 0.5) -- (4.5, 3.5);
\draw (3.5, 2.5) -- (3.5, 4.5);
\draw (2.5, 1.5) -- (2.5, 3.5);
\draw (0.5, 1.5) -- (1.4, 1.5);
\draw (2.5, 1.5) -- (1.6, 1.5);
\draw (1.5, 2.5) -- (2.4, 2.5);
\draw (3.5, 2.5) -- (2.6, 2.5);
\draw (2.5, 3.5) -- (3.4, 3.5);
\draw (4.5, 3.5) -- (3.6, 3.5);

\foreach \i\j in {0.5/1.5, 1.5/2.5, 2.5/3.5, 3.5/4.5, 4.5/0.5}{

\draw (\i - 0.1, \j + 0.1) -- (\i + 0.1, \j - 0.1);
\draw (\i - 0.1, \j - 0.1) -- (\i + 0.1, \j + 0.1);

}

\foreach \i\j in {0.5/4.5, 1.5/0.5, 2.5/1.5, 3.5/2.5, 4.5/3.5}{

\draw (\i, \j) circle (0.1);

}

\foreach \i\j in {0/4, 1/0, 2/1, 3/3, 4/2}{

\filldraw[black] (\i, \j) circle (0.1);

}

\end{tikzpicture}\caption{A $5 \times 5$ grid diagram for the trefoil, along with the generator $[51243]$ drawn with \textbullet. Note that the generator is independent of the $X$ and $O$ markings.}\label{fig1}\end{figure}
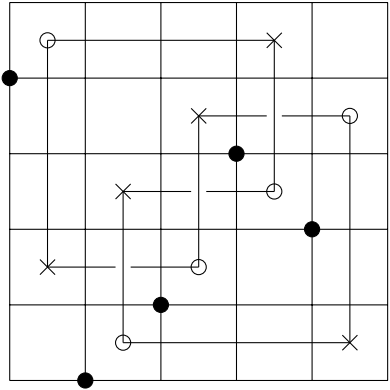

Grid diagrams have been useful in a variety of applications in Heegaard Floer homology. \cite{MOT} and \cite{MO} obtain the Heegaard Floer invariants of 3- and 4-manifolds using grid diagrams, which give algorithmically computable descriptions. \cite{Sar} uses grid homology to give another proof of Milnor's conjecture on the slice genus of torus knots. \cite{OST}, \cite{NOT}, \cite{CN}, and \cite{KN} use a version of grid homology to prove results about Legendrian knots.

\cite{MS} constructed a stable homotopy refinement of knot Floer homology from the grid chain complex, using framed flow categories in the sense of \cite{CJS}. The Manolescu-Sarkar construction uses only those domains that do not pass through a particular square on the grid, and uses obstruction theory. Their obstruction chain complexes $\CD_*$ and $\CDP_*$, which we will henceforth denote $\widehat{\CD}_*$ and $\widehat{\CDP}_*$, respectively, have simple enough homology to construct a stable homotopy type. We will extend them to complexes $\CD_*$ and $\CDP_*$ which contain all domains in the grid, first with $\Z/2$ coefficients, which can be extended to $\Z$ coefficients using obstruction theory. We take the first step towards extending the Manolescu-Sarkar construction, by computing the homology of $\CD_*$ and partially computing the homology of $\CDP_*$.

To state our main results, we fix the following convention throughout the paper. For a ring $R$, $R^{2^n}$ will denote the chain complex given by $R^{\binom{n}{k}}$ in grading $k$ with no differentials, and $R[U]$ the chain complex given by $R$ in every nonnegative even grading and $0$ in every odd grading (which by definition has no differentials). We begin by showing that:

\begin{prop}$H_*(\CD_*; \Z/2)$ is isomorphic to $\Z/2[U]$.\label{prop1}\end{prop}

In order to frame the moduli spaces in the Manolescu-Sarkar construction, we will need a sign assignment for the grid diagram. A sign assignment is a particular way of orienting the index $1$ domains in Heegaard Floer homology; equivalently, it is a particular assignment of $0$ or $1$ to each rectangle in the grid. The existence and uniqueness (up to gauge equivalence) of sign assignments for toroidal grid diagrams was constructed by \cite{MOST}; see also \cite{GA} for an explicit construction. In the course of our later computations, we will provide a different proof of this fact via obstruction theory:

\begin{theorem}Sign assignments for $\CD_*$ exist and are unique up to gauge equivalence (equivalently, up to $1$-coboundaries of $\CD_*$ with $\Z/2$ coefficients).\label{prop11}\end{theorem}

Given a sign assignment for $\CD_*$, we obtain a definition of $\CD_*$ in $\Z$ coefficients. Perhaps unsurprisingly, we then obtain the following analogue of Proposition $\ref{prop1}$:

\begin{prop}$H_*(\CD_*; \Z)$ is isomorphic to $\Z[U]$.\label{prop3}\end{prop}

In the end, our eventual goal is to extend the Manolescu-Sarkar construction over the full grid. Since the moduli spaces presented in \cite{MS} exhibit some bubbling, we will compute the first few homology groups of $\CDP_*$. Unfortunately, $\CDP_*$ has too much homology to immediately construct a stable homotopy type. So instead, we will work towards constructing a framed $1$-flow category, which is a formulation by \cite{LOS} that still contains all the information needed to define invariants such as the second Steenrod square. This requires only a sign assignment and a frame assignment, whose obstructions lie in the following lower homologies.

\begin{theorem}We have that
\begin{enumerate}

\item[(0)] $H_0(\CDP_*; \Z/2)$ is isomorphic to $\Z/2$.

\item[(1)] $H_1(\CDP_*; \Z/2)$ is isomorphic to $(\Z/2)^{n}$.

\item[(2)] $H_2(\CDP_*; \Z/2)$ is isomorphic to $(\Z/2)^{\binom{n}{2} + 1}$.

\item[(3)] $H_3(\CDP_*; \Z/2)$ is isomorphic to $(\Z/2)^{\binom{n}{3} + n}$.

\end{enumerate}\label{prop8}\end{theorem}

In this paper, we will show existence and uniqueness of sign assignments for $\CDP_*$.

\begin{theorem}A sign assignment $s$ on $\CDP_*$ exists, and is unique up to gauge transformations and the values of 
\begin{align*}s_j:= s(c_{x^{\Id}}, \vec{e}_j, (1)).\end{align*}\label{prop6}\end{theorem}

(The elements $(c_{x^{\Id}}, \vec{e}_j, (1)) \in \CDP_*$ will be defined later in Section $4$.)

Just like for $\CD_*$, we can use Theorem $\ref{prop6}$ to define $\CDP_*$ with $\Z$ coefficients. We have the following analogue of Theorem $\ref{prop8}$.







It remains to find a frame assignment for $\CDP_*$ using the above homology computation, and to complete the construction of the $1$-flow category for the full grid, which we will carry out in a future paper. This present paper may be treated as a prelude thereof.


\textbf{Acknowledgements.} The author would like to thank Sucharit Sarkar for many helpful conversations, as well as the referee for many helpful suggestions and corrections. This work was supported by an NSF grant DMS-2136090.

\section{The Obstruction Complex}

Definitions related to grid diagrams are summarized below. For details, see \cite{MOS, MOST, OSS}.

\begin{itemize}

\item An index $n$ grid diagram $\G$ is a torus together with $n$ $\alpha$-circles (drawn horizontally) and $n$ $\beta$-circles (drawn vertically). The complements of the $\alpha$ (respectively, $\beta$) circles are called the horizontal (respectively, vertical) annuli---the complements of the $\alpha$ and $\beta$ circles are called the square regions.

\item Each vertical and horizontal annulus contains exactly one $X$ and $O$ marking, which are arbitrarily labelled $X_1, \dots, X_n$ and $O_1, \dots, O_n$.

\item The horizontal (respectively, vertical) annuli can be labeled by which $O$-marking they pass through---write $H_i$ (respectively, $V_i$) for the horizontal (respectively, vertical) annulus passing through $O_i$.

\item Given a fixed planar drawing of the grid, we can also label the the $\alpha$ circles $\alpha_1, \dots, \alpha_n$ from bottom to top, and the $\beta$ circles $\beta_1, \dots, \beta_n$ from left to right. The annuli can also be labelled by which sets of $\alpha$ or $\beta$ circles they lie between---write $H_{(i)}$ (respectively, $V_{(i)}$) for the horizontal annulus between $\alpha_i$ and $\alpha_{i + 1}$ (respectively, vertical annulus between $\beta_i$ and $\beta_{i + 1}$). Note that $H_{(n)}$ and $V_{(n)}$ lie between $\alpha_n$ and $\alpha_1$, and $\beta_n$ and $\beta_1$, respectively.

\item A generator is an unordered $n$-tuple of points such that each $\alpha$ and $\beta$ circle contains exactly one. Generators can equivalently be viewed a $\Z$-linear combination of $n$ points, or alternatively as permutations---for a permutation $\sigma \in S_n$ the generator $x^{\sigma}$ is the unique generator with a point at each $\alpha_{\sigma(i)} \cap \beta_{i}$. In this paper we will use the convention that $[a_1 a_2\dots a_n]$ denotes the permutation $\sigma \in S_n$ where $\sigma(j) = a_j$ for each $j$. For instance, Figure $\ref{fig1}$ shows the generator $x^{[51243]}$, which we will interchangeably denote as $[51243]$.

\item A domain is a $\Z$-linear combination of square regions with the property that $\partial(\partial D \cap \alpha) = y - x$ for some generators $x, y$. We say that $D$ is a domain from $x$ to $y$, and write $D \in \mathscr{D}(x, y)$. $D$ is said to be positive if none of its coefficients are negative, in which case we would write $D \in \mathscr{D}^+(x, y)$.

\item Given $D \in \mathscr{D}(x, y), E \in \mathscr{D}(y, z)$, we get a domain $D * E \in \mathscr{D}(x, z)$ by adding $D$ and $E$ as $2$-chains.

\item The constant domain from a generator $x$ to itself is the domain $c_x \in \mathscr{D}(x, x)$ whose coefficients are zero in every square region.

\item For every domain $D$, there is an associated integer $\mu(D)$ called its Maslov index, which satisfies: \begin{itemize}\item $\mu(D * E) = \mu(D) + \mu(E)$ \item For a positive domain $D$, $\mu(D) \geq 0$, with equality if and only if $D$ is some constant domain. \item For $D \in \mathscr{D}^+(x, y)$, $\mu(D) = 1$ if and only if $D$ is a rectangle: that is, its bottom left and top right corners are coordinates of $x$, its bottom right and top left corners are coordinates of $y$, and the other $n - 2$ coordinates of $x$ and $y$ agree and do not lie in $D$. \item $\mu(D) = k$ if and only if $D$ can be decomposed (not necessarily uniquely) into $k$ rectangles $D = R_1 * \cdots * R_k$.\end{itemize}

\end{itemize}

It will be particularly helpful to classify positive index $2$ domains, which are exactly those that can be decomposed as two rectangles. Thus every positive index $2$ domain $D \in \mathscr{D}^+(x, y)$ is a horizontal or vertical annulus\\
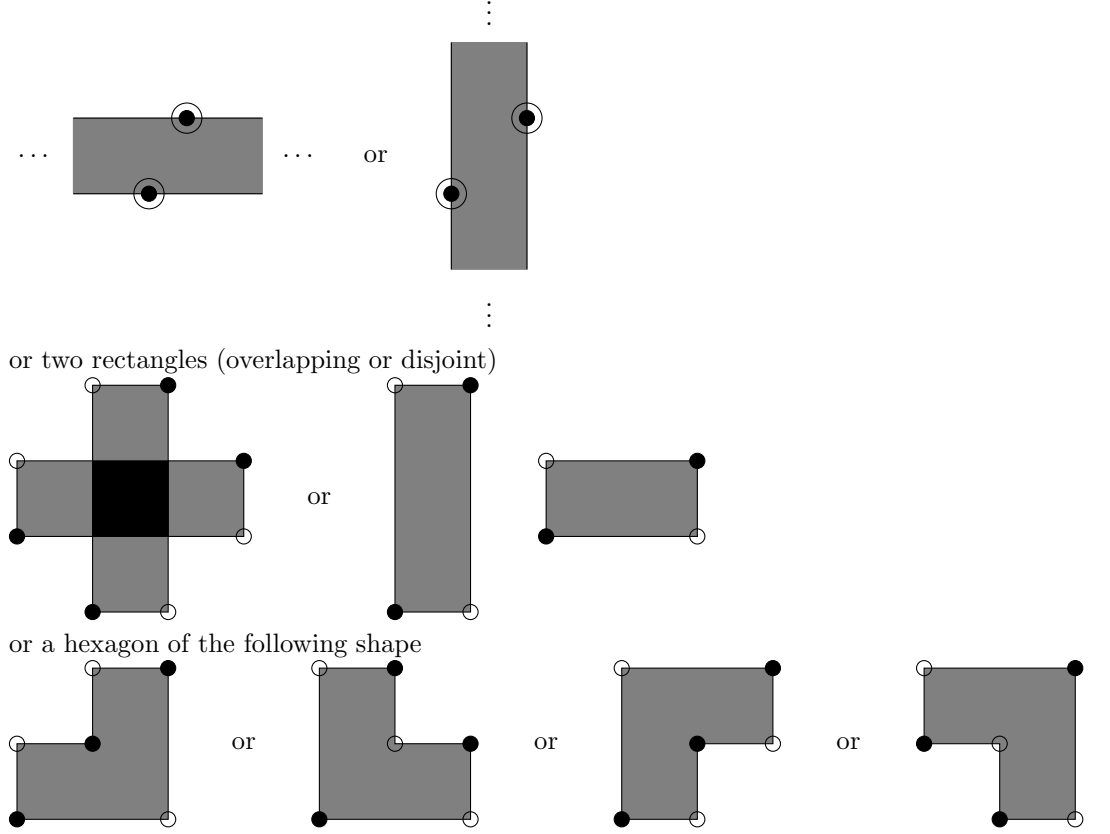
\begin{tikzpicture}

\filldraw[gray] (1, 0) rectangle (3.5, 1);
\draw (1, 0) -- (3.5, 0);
\draw (1, 1) -- (3.5, 1);

\node[] at (0.5, 0.5) {$\dots$};
\node[] at (4, 0.5) {$\dots$};

\foreach \i\j in {2/0, 2.5/1}{

\filldraw[black] (\i, \j) circle (0.1);

\draw (\i, \j) circle (0.2);

}

\node[] at (5, 0.5) {$\text{or}$};

\filldraw[gray] (6, -1) rectangle (7, 2);
\draw (6, -1) -- (6, 2);
\draw (7, -1) -- (7, 2);

\node[] at (6.5, 2.5) {$\vdots$};
\node[] at (6.5, -1.5) {$\vdots$};

\foreach \i\j in {6/0, 7/1}{

\filldraw[black] (\i, \j) circle (0.1);

\draw (\i, \j) circle (0.2);

}

\end{tikzpicture}
\\or two rectangles (overlapping or disjoint)\\
\begin{tikzpicture}

\draw[fill=gray] (0, 0) rectangle (3, 1);
\draw[fill=gray] (1, -1) rectangle (2, 2);
\draw[fill=black] (1, 0) rectangle (2, 1);

\foreach \i\j in {0/0, 1/-1, 2/2, 3/1}{

\filldraw[black] (\i, \j) circle (0.1);

}

\foreach \i\j in {0/1, 1/2, 2/-1, 3/0}{

\draw (\i, \j) circle (0.1);

}

\node[] at (4, 0.5) {$\text{or}$};

\draw[fill=gray] (5, -1) rectangle (6, 2);
\draw[fill=gray] (7, 0) rectangle (9, 1);

\foreach \i\j in {5/-1, 6/2, 7/0, 9/1}{

\filldraw[black] (\i, \j) circle (0.1);

}

\foreach \i\j in {5/2, 6/-1, 7/1, 9/0}{

\draw (\i, \j) circle (0.1);

}

\end{tikzpicture}\\
or a hexagon of the following shape\\
\begin{tikzpicture}

\draw[fill=gray] (0, 0) -- (2, 0) -- (2, 2) -- (1, 2) -- (1, 1) -- (0, 1) -- (0, 0);
\draw[fill=gray] (4, 0) -- (6, 0) -- (6, 1) -- (5, 1) -- (5, 2) -- (4, 2) -- (4, 0);
\draw[fill=gray] (8, 0) -- (8, 2) -- (10, 2) -- (10, 1) -- (9, 1) -- (9, 0) -- (8, 0);
\draw[fill=gray] (12, 1) -- (12, 2) -- (14, 2) -- (14, 0) -- (13, 0) -- (13, 1) -- (12, 1);

\node[] at (3, 1) {$\text{or}$};
\node[] at (7, 1) {$\text{or}$};
\node[] at (11, 1) {$\text{or}$};
\node[] at (14.5, 1) {.};

\foreach \i\j in {0/0, 1/1, 2/2, 4/0, 5/2, 6/1, 8/0, 9/1, 10/2, 12/1, 13/0, 14/2}{

\filldraw[black] (\i, \j) circle (0.1);

}

\foreach \i\j in {0/1, 1/2, 2/0, 4/2, 5/1, 6/0, 8/2, 9/0, 10/1, 12/2, 13/1, 14/0}{

\draw (\i, \j) circle (0.1);

}

\end{tikzpicture}

(Here the generator $x$ is shown by \textbullet while $y$ is shown by $\circ$.) Note that while a horizontal or vertical annulus admits exactly one decomposition into rectangles, all the other positive index $2$ domains admit exactly two.

Given a grid diagram $\G$, we define the complex of positive domains, on which our desired sign assignment can be constructed as a cochain.

\begin{definition}The complex of positive domains $\CD_* = \CD_*(\G; \Z/2)$ is freely generated over $\Z/2$ by the positive domains, with the homological grading being the Maslov index:
\begin{align*}\CD_k = \Z/2\langle \{ (x, y, D) \: | \: D \in \mathscr{D}^+(x, y), \mu(D) = k \} \rangle. \end{align*}
Sometimes the generators $x, y$ will be omitted. The differential $\partial: \CD_k \rightarrow \CD_{k - 1}$ of $D \in \mathscr{D}^+(x, y)$ is given by
\begin{align*}\partial(D) = \sum\limits_{R * E = D} E + \sum\limits_{E * R = D} E,\end{align*}
where $R$ is a rectangle, and $E$ is a positive domain.\label{def1}\end{definition}

Note that $\CD_*$ is independent of the placement of the $X$'s and $O$'s.

\begin{lemma}$(\CD_*, \partial)$ is a chain complex, that is, $\partial^2 = 0$. \label{lem1}\end{lemma}

\begin{proof}Let $R$ and $S$ denote rectangles, then
\begin{align*}\partial^2(D) = \sum\limits_{R * S * E = D} E + \sum\limits_{R * E * S = D} E + \sum\limits_{S * E * R = D} E + \sum\limits_{E * S * R = D} E\end{align*}
The second and third terms cancel (modulo $2$). If $R * S$ is a hexagon or two rectangles, then it has exactly one other decomposition $R * S = R' * S'$, so $R * S * E$ and $R' * S' * E$ cancel in the first term. Similarly, $E * R * S$ and $E * R' * S'$ cancel in the last term. Finally, if $R * S$ is not a hexagon or two rectangles, it must be a horizontal or vertical annulus, and then the terms $E * R * S$ and $R * S * E$ in the first and last term cancel, and so $\partial^2(D) = 0$.\end{proof}

We now compute the homology $\CD_*$ by constructing filtrations, for which we need the following fact. Given two generators $x$ and $y$, we say that $x \leq y$ if there exists a positive domain from $y$ to $x$ that does not intersect the topmost row $H_{(n)}$ or rightmost column $V_{(n)}$ of the grid. It is clear that the set of generators with $\leq$ is a partially ordered set (which actually coincides with the opposite of the Bruhat ordering on the symmetric group $S_n$---see \cite[Section 3.2]{MS}).



\begin{proof}[Proof of Proposition \ref{prop1}] The proof is nearly identical to the proof of \cite[Proposition 3.4]{MS}, so we present the most relevant parts. To $D \in \CD_*$ associate $A(D) \in \N^n$ by its coefficients in the rightmost vertical annulus. Note that here, unlike in \cite{MS}, $A(D)$ is an $n$-tuple, since there is no assumption that domains do not pass through the top right corner. By definition, the differential only preserves or lowers $A(D)$, so it is a filtration on $\CD_*$. Now let $\CD_*^{a}$ be the associated graded complex in filtration grading $A(D) = a$.

Let $M(D) = \min \{ \text{coordinates of } A(D) \}$---by definition, a positive domain $D$ contains exactly $M(D)$ copies of the rightmost vertical annulus $V_{(n)}$, so write $D = D' * M(D) V_{(n)}$. $A(D')$ contains a $0$, so without loss of generality (since the differential of $\CD_*^{a}$ does not change $A(D)$ and thus does not change where the $0$ is located) $D'$ does not contain the top right corner. Now let $B(D) \in \N^{n - 1}$ be the coordinates of $D'$ in the top row (except the top right corner). Similarly, $B(D)$ is a filtration on the associated graded complex $\CD_*^{a}$, so let $\CD_*^{a, b}$ be the associated graded complex in grading $A(D) = a, B(D) = b$.

Now fix $(a, b)$ and consider the differential $\partial$ on $\CD_*^{a, b}$. Consider the following new filtration. For any domain $D \in \mathscr{D}^+(x, y)$ with $A(D) = a, B(D) = b$, let the generator $y$ be its filtration grading. With respect to the aforementioned partial ordering of the generators, $\partial$ preserves or decreases $y$ since we only consider removing domains that do not pass through the topmost row and rightmost column. Therefore $y$ is a filtration grading, so let $\CD_*^{a, b, y}$ be the associated graded complex with respect to this filtration. Unless $a = (l, l, \dots, l)$, $b = 0$, and $y = x^{\Id}$, the proof of \cite{MS} shows that $\CD_*^{a, b, y}$ is acyclic. When $a = (l, l, \dots, l)$, $b = 0$, and $y = x^{\Id}$, the complex $\CD_*^{a, b, y}$ has one generator (since $x^{\Id}$ is maximal), which is represented by the domain $lV_{(n)}$, lying in grading $2l$. 

Finally, because the associated graded complex has homology only in even gradings, $\CD_*$ must have the same homology.\end{proof}


In order to later remove obstructions in grading $2$, we now explicitly find the generator $U$ of $H_2(\CD_*)$. We define the following index $2$ domains:
\begin{itemize}
\item $A_1, \dots, A_{n - 1}$ where $A_i$ is the vertical annulus in the $(n-i)^{th}$ column from the left from the generator $[n23\dots (n-i)1(n-i+1)\dots (n-1)]$ to itself, and $A_0$ is the rightmost vertical annulus from the identity generator $x^{\Id}$ to itself.

\item $B_1, \dots, B_{n - 1}$ where $B_i$ is the horizontal annulus in the $(n-i)^{th}$ row from the bottom from the generator $[(n-i+1)23\dots (n-i)(n-i+2)\dots n1]$ to itself, and $B_0$ is the topmost horizontal annulus from the identity generator $x^{\Id}$ to itself.

\item $C_1, \dots, C_{n - 2}$ where $C_i$ is a hexagon from the generator $[n23\dots (n-i)1(n-i+1)\dots (n - 1)]$ to the generator $[12\dots(n-i-1)n(n-i)\dots(n-1)]$.

\item $D_1, \dots, D_{n - 2}$ where $D_i$ is a hexagon from the generator $[(n - i + 1)23\dots (n-i)(n - i + 2)\dots n1]$ to the generator $[12\dots(n-i-1)(n-i + 1)\dots n (n - i)]$.

\item $E_1, \dots, E_{n - 2}$ where $E_i$ is a hexagon from the generator $[12\dots (n - i - 1)n(n-i+1)\dots (n-1)(n-i)]$ to the generator $[12\dots (n-i-2)n(n-i)\dots (n-1)(n-i-1)]$.

\item $F_{i, 1}, \dots, F_{i, n - i - 2}$ for each $i = 1, \dots n - 3$, where $F_{i, j}$ is a hexagon from the generator $[12\dots (n - i - j - 2)(n-i-j)(n-i)(n-i-j+1)\dots (n-i-1) (n - i + 1) \dots n(n-i-j-1)]$ to the generator $[12\dots (n - i - j - 2)(n-i + 1)(n-i-j)(n-i-j+1)\dots (n-i) (n - i + 2) \dots n(n-i-j-1)]$.

\item $G_{i, 1}, \dots, G_{i, n - i - 2}$ for each $i = 1, \dots n - 3$, where $G_{i, j}$ is a hexagon from the generator $[12\dots (n - i - j - 2)n(n-i-j-1)(n-i-j+1)\dots (n-i-1) (n - i - j)(n - i) \dots (n - 1)]$ to the generator $[12\dots (n - i - j - 2)n(n-i-j)\dots(n - i) (n - i - j - 1)(n - i + 1)\dots (n - 1)]$.

\end{itemize}
(see Figure $\ref{fig2}$)

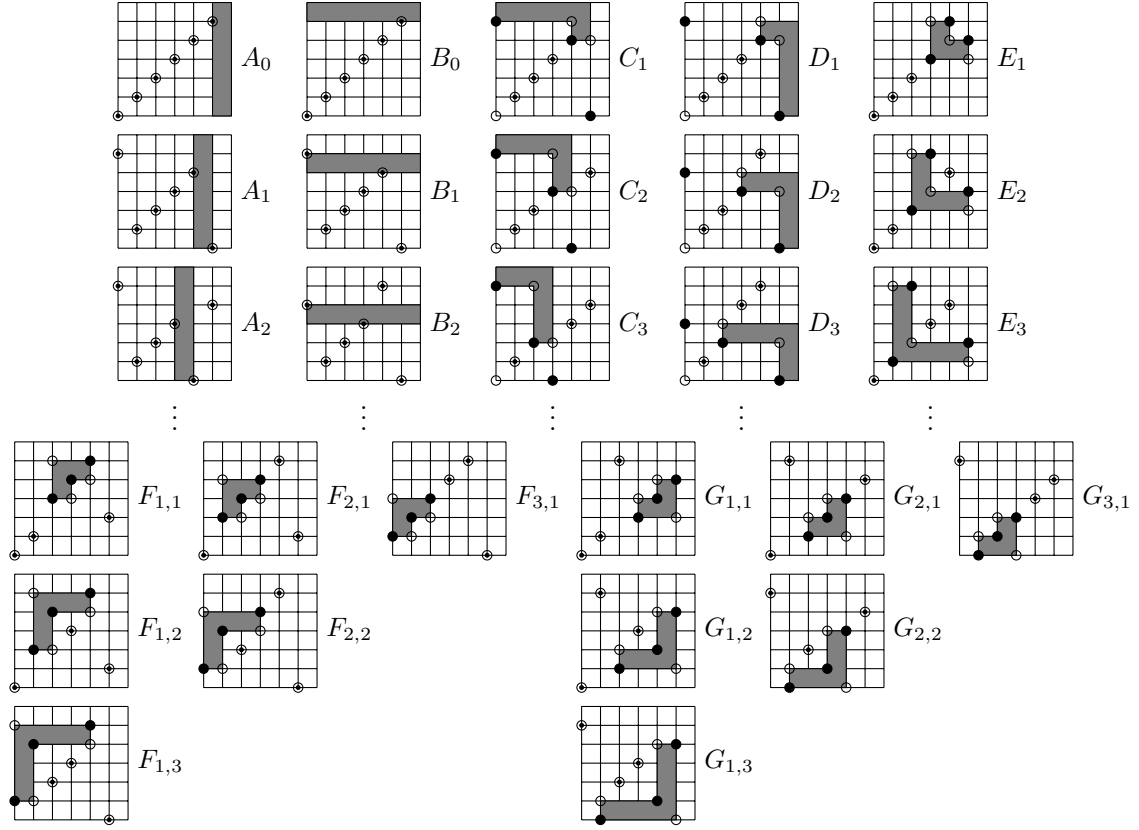
\begin{figure}\begin{tikzpicture}[scale=0.25]

\foreach \i in {0,1,...,4}{
    \foreach \j in {0,1,2}{

      \begin{scope}[xshift=10*\i cm,yshift=-7*\j cm] 

        \draw (0,0) grid (6,6);
        
      \end{scope}

}

\node[anchor=north] at (3 + 10*\i, -14) {\vdots};

}

  \begin{scope}[xshift=0cm, yshift=0 cm]
    \node[anchor=west] at (6,3) {$A_0$};

    \draw[fill=gray] (5,0) rectangle (6,6);

    \foreach \a/\b in {0/0, 1/1, 2/2, 3/3, 4/4, 5/5}{

      \filldraw[black] (\a,\b) circle (0.1);
      \draw (\a,\b) circle (0.25);

    }
    
  \end{scope}

  \begin{scope}[xshift=0cm, yshift=-7 cm]
    \node[anchor=west] at (6,3) {$A_1$};

    \draw[fill=gray] (4,0) rectangle (5,6);

    \foreach \a/\b in {0/5,1/1,2/2,3/3,4/4,5/0}{

      \filldraw[black] (\a,\b) circle (0.1);
      \draw (\a,\b) circle (0.25);

    }
    
  \end{scope}

  \begin{scope}[xshift=0cm, yshift=-14 cm]
    \node[anchor=west] at (6,3) {$A_2$};

    \draw[fill=gray] (3,0) rectangle (4,6);

    \foreach \a/\b in {0/5, 1/1, 2/2, 3/3, 4/0, 5/4}{

      \filldraw[black] (\a,\b) circle (0.1);
      \draw (\a,\b) circle (0.25);

    }
    
  \end{scope}

    \begin{scope}[xshift=10cm, yshift=0 cm]
    \node[anchor=west] at (6,3) {$B_0$};

    \draw[fill=gray] (0,5) rectangle (6,6);

    \foreach \a/\b in {0/0, 1/1, 2/2, 3/3, 4/4, 5/5}{

      \filldraw[black] (\a,\b) circle (0.1);
      \draw (\a,\b) circle (0.25);

    }
    
  \end{scope}

  \begin{scope}[xshift=10cm, yshift=-7 cm]
    \node[anchor=west] at (6,3) {$B_1$};

    \draw[fill=gray] (0,4) rectangle (6,5);

    \foreach \a/\b in {0/5,1/1,2/2,3/3,4/4,5/0}{

      \filldraw[black] (\a,\b) circle (0.1);
      \draw (\a,\b) circle (0.25);

    }
    
  \end{scope}

  \begin{scope}[xshift=10cm, yshift=-14 cm]
    \node[anchor=west] at (6,3) {$B_2$};

    \draw[fill=gray] (0,3) rectangle (6,4);

    \foreach \a/\b in {0/4, 1/1, 2/2, 3/3, 4/5, 5/0}{

      \filldraw[black] (\a,\b) circle (0.1);
      \draw (\a,\b) circle (0.25);

    }
    
  \end{scope}

    \begin{scope}[xshift=20cm, yshift=0 cm]
    \node[anchor=west] at (6,3) {$C_1$};

    \draw[fill=gray] (0, 5) to (0, 6) to (5, 6) to (5, 4) to (4, 4) to (4, 5) to (0, 5);

    \foreach \a/\b in {1/1, 2/2, 3/3}{

      \filldraw[black] (\a,\b) circle (0.1);
      \draw (\a,\b) circle (0.25);

    }

    \foreach \a/\b in {0/5, 4/4, 5/0}{

      \filldraw[black] (\a,\b) circle (0.25);

    }

    \foreach \a/\b in {0/0, 4/5, 5/4}{

      \draw (\a,\b) circle (0.25);

    }
    
  \end{scope}

  \begin{scope}[xshift=20cm, yshift=-7 cm]
    \node[anchor=west] at (6,3) {$C_2$};

    \draw[fill=gray] (0, 5) to (0, 6) to (4, 6) to (4, 3) to (3, 3) to (3, 5) to (0, 5);

    \foreach \a/\b in {1/1, 2/2, 5/4}{

      \filldraw[black] (\a,\b) circle (0.1);
      \draw (\a,\b) circle (0.25);

    }

    \foreach \a/\b in {0/5, 3/3, 4/0}{

      \filldraw[black] (\a,\b) circle (0.25);

    }

    \foreach \a/\b in {0/0, 3/5, 4/3}{

      \draw (\a,\b) circle (0.25);

    }
    
  \end{scope}

  \begin{scope}[xshift=20cm, yshift=-14 cm]
    \node[anchor=west] at (6,3) {$C_3$};

    \draw[fill=gray] (0, 5) to (0, 6) to (3, 6) to (3, 2) to (2, 2) to (2, 5) to (0, 5);

    \foreach \a/\b in {1/1, 4/3, 5/4}{

      \filldraw[black] (\a,\b) circle (0.1);
      \draw (\a,\b) circle (0.25);

    }

    \foreach \a/\b in {0/5, 2/2, 3/0}{

      \filldraw[black] (\a,\b) circle (0.25);

    }

    \foreach \a/\b in {0/0, 2/5, 3/2}{

      \draw (\a,\b) circle (0.25);

    }
    
  \end{scope}

    \begin{scope}[xshift=30cm, yshift=0 cm]
    \node[anchor=west] at (6,3) {$D_1$};

    \draw[fill=gray] (5, 0) to (5, 4) to (4, 4) to (4, 5) to (6, 5) to (6, 0) to (5, 0);

    \foreach \a/\b in {1/1, 2/2, 3/3}{

      \filldraw[black] (\a,\b) circle (0.1);
      \draw (\a,\b) circle (0.25);

    }

    \foreach \a/\b in {0/5, 4/4, 5/0}{

      \filldraw[black] (\a,\b) circle (0.25);

    }

    \foreach \a/\b in {0/0, 4/5, 5/4}{

      \draw (\a,\b) circle (0.25);

    }
    
  \end{scope}

  \begin{scope}[xshift=30cm, yshift=-7 cm]
    \node[anchor=west] at (6,3) {$D_2$};

    \draw[fill=gray] (5, 0) to (5, 3) to (3, 3) to (3, 4) to (6, 4) to (6, 0) to (5, 0);

    \foreach \a/\b in {1/1, 2/2, 4/5}{

      \filldraw[black] (\a,\b) circle (0.1);
      \draw (\a,\b) circle (0.25);

    }

    \foreach \a/\b in {0/4, 3/3, 5/0}{

      \filldraw[black] (\a,\b) circle (0.25);

    }

    \foreach \a/\b in {0/0, 3/4, 5/3}{

      \draw (\a,\b) circle (0.25);

    }
    
  \end{scope}

  \begin{scope}[xshift=30cm, yshift=-14 cm]
    \node[anchor=west] at (6,3) {$D_3$};

   \draw[fill=gray] (5, 0) to (5, 2) to (2, 2) to (2, 3) to (6, 3) to (6, 0) to (5, 0);

    \foreach \a/\b in {1/1, 3/4, 4/5}{

      \filldraw[black] (\a,\b) circle (0.1);
      \draw (\a,\b) circle (0.25);

    }

    \foreach \a/\b in {0/3, 2/2, 5/0}{

      \filldraw[black] (\a,\b) circle (0.25);

    }

    \foreach \a/\b in {0/0, 2/3, 5/2}{

      \draw (\a,\b) circle (0.25);

    }
    
  \end{scope}

    \begin{scope}[xshift=40cm, yshift=0 cm]
    \node[anchor=west] at (6,3) {$E_1$};

    \draw[fill=gray] (3, 5) to (4, 5) to (4, 4) to (5, 4) to (5, 3) to (3, 3) to (3, 5);

    \foreach \a/\b in {0/0, 1/1, 2/2}{

      \filldraw[black] (\a,\b) circle (0.1);
      \draw (\a,\b) circle (0.25);

    }

    \foreach \a/\b in {3/3, 4/5, 5/4}{

      \filldraw[black] (\a,\b) circle (0.25);

    }

    \foreach \a/\b in {3/5, 4/4, 5/3}{

      \draw (\a,\b) circle (0.25);

    }
    
  \end{scope}

  \begin{scope}[xshift=40cm, yshift=-7 cm]
    \node[anchor=west] at (6,3) {$E_2$};

    \draw[fill=gray] (2, 5) to (3, 5) to (3, 3) to (5, 3) to (5, 2) to (2, 2) to (2, 5);

    \foreach \a/\b in {0/0, 1/1, 4/4}{

      \filldraw[black] (\a,\b) circle (0.1);
      \draw (\a,\b) circle (0.25);

    }

    \foreach \a/\b in {2/2, 3/5, 5/3}{

      \filldraw[black] (\a,\b) circle (0.25);

    }

    \foreach \a/\b in {2/5, 3/3, 5/2}{

      \draw (\a,\b) circle (0.25);

    }
    
  \end{scope}

  \begin{scope}[xshift=40cm, yshift=-14 cm]
    \node[anchor=west] at (6,3) {$E_3$};

    \draw[fill=gray] (1, 5) to (2, 5) to (2, 2) to (5, 2) to (5, 1) to (1, 1) to (1, 5);

    \foreach \a/\b in {0/0, 3/3, 4/4}{

      \filldraw[black] (\a,\b) circle (0.1);
      \draw (\a,\b) circle (0.25);

    }

    \foreach \a/\b in {1/1, 2/5, 5/2}{

      \filldraw[black] (\a,\b) circle (0.25);

    }

    \foreach \a/\b in {1/5, 2/2, 5/1}{

      \draw (\a,\b) circle (0.25);

    }
    \end{scope}

\end{tikzpicture}

\begin{tikzpicture}[scale=0.25]

  \begin{scope}[xshift=0cm,yshift=0cm] 

        \draw (0, 0) grid (6, 6);
        \node[anchor=west] at (6, 3){$F_{1, 1}$};
        
      \end{scope}

\begin{scope}[xshift=0cm,yshift=-7cm] 

       \draw (0, 0) grid (6, 6);
       \node[anchor=west] at (6, 3){$F_{1, 2}$};
        
      \end{scope}

\begin{scope}[xshift=0cm,yshift=-14cm] 

      \draw (0, 0) grid (6, 6);
      \node[anchor=west] at (6, 3){$F_{1, 3}$};
        
      \end{scope}

 \begin{scope}[xshift=0cm, yshift=0cm]

    \draw[fill=gray] (2, 3) -- (2, 5) -- (4, 5) -- (4, 4) -- (3, 4) -- (3, 3) -- (2, 3);

    \foreach \a/\b in {0/0, 1/1, 5/2}{

      \filldraw[black] (\a,\b) circle (0.1);
      \draw (\a,\b) circle (0.25);

    }

    \foreach \a/\b in {2/3, 3/4, 4/5}{

      \filldraw[black] (\a,\b) circle (0.25);

    }

    \foreach \a/\b in {2/5, 3/3, 4/4}{

      \draw (\a,\b) circle (0.25);

    }
    
  \end{scope}

  \begin{scope}[xshift=0cm, yshift=-7cm]

    \draw[fill=gray] (1, 2) -- (1, 5) -- (4, 5) -- (4, 4) -- (2, 4) -- (2, 2) -- (1, 2);

    \foreach \a/\b in {0/0, 3/3, 5/1}{

      \filldraw[black] (\a,\b) circle (0.1);
      \draw (\a,\b) circle (0.25);

    }

    \foreach \a/\b in {1/2, 2/4, 4/5}{

      \filldraw[black] (\a,\b) circle (0.25);

    }

    \foreach \a/\b in {1/5, 2/2, 4/4}{

      \draw (\a,\b) circle (0.25);

    }
    
  \end{scope}

  \begin{scope}[xshift=0cm, yshift=-14cm]

    \draw[fill=gray] (0, 1) -- (0, 5) -- (4, 5) -- (4, 4) -- (1, 4) -- (1, 1) -- (0, 1);

    \foreach \a/\b in {2/2, 3/3, 5/0}{

      \filldraw[black] (\a,\b) circle (0.1);
      \draw (\a,\b) circle (0.25);

    }

    \foreach \a/\b in {0/1, 1/4, 4/5}{

      \filldraw[black] (\a,\b) circle (0.25);

    }

    \foreach \a/\b in {0/5, 1/1, 4/4}{

      \draw (\a,\b) circle (0.25);

    }
    
  \end{scope}

  \begin{scope}[xshift=10cm,yshift=0cm] 

        \draw (0, 0) grid (6, 6);
        \node[anchor=west] at (6, 3){$F_{2, 1}$};
        
      \end{scope}

\begin{scope}[xshift=10cm,yshift=-7cm] 

       \draw (0, 0) grid (6, 6);
       \node[anchor=west] at (6, 3){$F_{2, 2}$};
        
      \end{scope}

 \begin{scope}[xshift=10cm, yshift=0cm]

    \draw[fill=gray] (1, 2) -- (1, 4) -- (3, 4) -- (3, 3) -- (2, 3) -- (2, 2) -- (1, 2);

    \foreach \a/\b in {0/0, 4/5, 5/1}{

      \filldraw[black] (\a,\b) circle (0.1);
      \draw (\a,\b) circle (0.25);

    }

    \foreach \a/\b in {1/2, 2/3, 3/4}{

      \filldraw[black] (\a,\b) circle (0.25);

    }

    \foreach \a/\b in {1/4, 2/2, 3/3}{

      \draw (\a,\b) circle (0.25);

    }
    
  \end{scope}

  \begin{scope}[xshift=10cm, yshift=-7cm]

    \draw[fill=gray] (0, 1) -- (0, 4) -- (3, 4) -- (3, 3) -- (1, 3) -- (1, 1) -- (0, 1);

    \foreach \a/\b in {2/2, 4/5, 5/0}{

      \filldraw[black] (\a,\b) circle (0.1);
      \draw (\a,\b) circle (0.25);

    }

    \foreach \a/\b in {0/1, 1/3, 3/4}{

      \filldraw[black] (\a,\b) circle (0.25);

    }

    \foreach \a/\b in {0/4, 1/1, 3/3}{

      \draw (\a,\b) circle (0.25);

    }
    
  \end{scope}

  \begin{scope}[xshift=20cm,yshift=0cm] 

        \draw (0, 0) grid (6, 6);
        \node[anchor=west] at (6, 3){$F_{3, 1}$};
        
      \end{scope}

 \begin{scope}[xshift=20cm, yshift=0cm]

    \draw[fill=gray] (0, 1) -- (0, 3) -- (2, 3) -- (2, 2) -- (1, 2) -- (1, 1) -- (0, 1);

    \foreach \a/\b in {3/4, 4/5, 5/0}{

      \filldraw[black] (\a,\b) circle (0.1);
      \draw (\a,\b) circle (0.25);

    }

    \foreach \a/\b in {0/1, 1/2, 2/3}{

      \filldraw[black] (\a,\b) circle (0.25);

    }

    \foreach \a/\b in {0/3, 1/1, 2/2}{

      \draw (\a,\b) circle (0.25);

    }
    
  \end{scope}

\begin{scope}[xshift=30cm,yshift=0cm] 

        \draw (0, 0) grid (6, 6);
        \node[anchor=west] at (6, 3){$G_{1, 1}$};
        
      \end{scope}

\begin{scope}[xshift=30cm,yshift=-7cm] 

       \draw (0, 0) grid (6, 6);
       \node[anchor=west] at (6, 3){$G_{1, 2}$};
        
      \end{scope}

\begin{scope}[xshift=30cm,yshift=-14cm] 

      \draw (0, 0) grid (6, 6);
      \node[anchor=west] at (6, 3){$G_{1, 3}$};
        
      \end{scope}

 \begin{scope}[xshift=30cm, yshift=0cm]

    \draw[fill=gray] (3, 2) -- (5, 2) -- (5, 4) -- (4, 4) -- (4, 3) -- (3, 3) -- (3, 2);

    \foreach \a/\b in {0/0, 1/1, 2/5}{

      \filldraw[black] (\a,\b) circle (0.1);
      \draw (\a,\b) circle (0.25);

    }

    \foreach \a/\b in {3/2, 4/3, 5/4}{

      \filldraw[black] (\a,\b) circle (0.25);

    }

    \foreach \a/\b in {5/2, 3/3, 4/4}{

      \draw (\a,\b) circle (0.25);

    }
    
  \end{scope}

  \begin{scope}[xshift=30cm, yshift=-7cm]

    \draw[fill=gray] (2, 1) -- (5, 1) -- (5, 4) -- (4, 4) -- (4, 2) -- (2, 2) -- (2, 1);

    \foreach \a/\b in {0/0, 3/3, 1/5}{

      \filldraw[black] (\a,\b) circle (0.1);
      \draw (\a,\b) circle (0.25);

    }

    \foreach \a/\b in {2/1, 4/2, 5/4}{

      \filldraw[black] (\a,\b) circle (0.25);

    }

    \foreach \a/\b in {5/1, 2/2, 4/4}{

      \draw (\a,\b) circle (0.25);

    }
    
  \end{scope}

  \begin{scope}[xshift=30cm, yshift=-14cm]

    \draw[fill=gray] (1, 0) -- (5, 0) -- (5, 4) -- (4, 4) -- (4, 1) -- (1, 1) -- (1, 0);

    \foreach \a/\b in {2/2, 3/3, 0/5}{

      \filldraw[black] (\a,\b) circle (0.1);
      \draw (\a,\b) circle (0.25);

    }

    \foreach \a/\b in {1/0, 4/1, 5/4}{

      \filldraw[black] (\a,\b) circle (0.25);

    }

    \foreach \a/\b in {5/0, 1/1, 4/4}{

      \draw (\a,\b) circle (0.25);

    }
    
  \end{scope}

  \begin{scope}[xshift=40cm,yshift=0cm] 

        \draw (0, 0) grid (6, 6);
        \node[anchor=west] at (6, 3){$G_{2, 1}$};
        
      \end{scope}

\begin{scope}[xshift=40cm,yshift=-7cm] 

       \draw (0, 0) grid (6, 6);
       \node[anchor=west] at (6, 3){$G_{2, 2}$};
        
      \end{scope}

 \begin{scope}[xshift=40cm, yshift=0cm]

    \draw[fill=gray] (2, 1) -- (4, 1) -- (4, 3) -- (3, 3) -- (3, 2) -- (2, 2) -- (2, 1);

    \foreach \a/\b in {0/0, 5/4, 1/5}{

      \filldraw[black] (\a,\b) circle (0.1);
      \draw (\a,\b) circle (0.25);

    }

    \foreach \a/\b in {2/1, 3/2, 4/3}{

      \filldraw[black] (\a,\b) circle (0.25);

    }

    \foreach \a/\b in {4/1, 2/2, 3/3}{

      \draw (\a,\b) circle (0.25);

    }
    
  \end{scope}

  \begin{scope}[xshift=40cm, yshift=-7cm]

    \draw[fill=gray] (1, 0) -- (4, 0) -- (4, 3) -- (3, 3) -- (3, 1) -- (1, 1) -- (1, 0);

    \foreach \a/\b in {2/2, 5/4, 0/5}{

      \filldraw[black] (\a,\b) circle (0.1);
      \draw (\a,\b) circle (0.25);

    }

    \foreach \a/\b in {1/0, 3/1, 4/3}{

      \filldraw[black] (\a,\b) circle (0.25);

    }

    \foreach \a/\b in {4/0, 1/1, 3/3}{

      \draw (\a,\b) circle (0.25);

    }
    
  \end{scope}

  \begin{scope}[xshift=50cm,yshift=0cm] 

        \draw (0, 0) grid (6, 6);
        \node[anchor=west] at (6, 3){$G_{3, 1}$};
        
      \end{scope}

 \begin{scope}[xshift=50cm, yshift=0cm]

    \draw[fill=gray] (1, 0) -- (3, 0) -- (3, 2) -- (2, 2) -- (2, 1) -- (1, 1) -- (1, 0);

    \foreach \a/\b in {4/3, 5/4, 0/5}{

      \filldraw[black] (\a,\b) circle (0.1);
      \draw (\a,\b) circle (0.25);

    }

    \foreach \a/\b in {1/0, 2/1, 3/2}{

      \filldraw[black] (\a,\b) circle (0.25);

    }

    \foreach \a/\b in {3/0, 1/1, 2/2}{

      \draw (\a,\b) circle (0.25);

    }
    
  \end{scope}

\end{tikzpicture}\caption{The domains $A_i, B_i, C_i, D_i, E_i, F_{i, j}$, and $G_{i, j}$ in the special case of a $6 \times 6$ grid, where each domain is drawn from a generator $x$ (drawn by \textbullet) to a generator $y$ (drawn by $\circ$)}. \label{fig2}\end{figure}

Let
\begin{align*}U := \sum\limits_{i = 0}^{n - 1} (A_i + B_i) + \sum\limits_{i = 1}^{n - 2} (C_i + D_i) + \sum\limits_{i = 1}^{n - 2} E_i + \sum\limits_{i = i}^{n - 3} \sum\limits_{j = 1}^{n - i - 2} (F_{i, j} + G_{i, j})\end{align*}

\begin{prop}$U$ is the generator of $H_2(\CD_*)$\label{prop2}\end{prop}

Proposition $\ref{prop2}$ will follow from the following computational lemmas.

\begin{lemma}$U$ is a cycle in $\CD_2$ (that is, $\partial U = 0$). \label{lem7}\end{lemma}

\begin{proof} We consider the possible rectangles that appear in $\partial U$, starting with the following rectangles that will be useful to name for the purposes of giving signs later.
\begin{itemize}

\item $R_{1,2}, \dots, R_{1,n - 1}$ where $R_{1,i}$ is the $1 \times i$ rectangle from $[n23\dots (n - i + 1)1(n - i + 2)\dots(n - 1)]$ to $[n23\dots (n - i)1(n - i + 1)\dots(n - 1)]$, and $R_{1,1}$ is the $1 \times 1$ rectangle from $x^{\Id}$ to $[n23\dots(n-1) 1]$.

\item $R_{2,2}, \dots, R_{2,n - 1}$ where $R_{2,i}$ is the $1 \times i$ rectangle from $[12\dots (n - i)n(n - i + 1) \dots(n - 1)]$ to $[12 \dots (n - i - 1)n(n - i) \dots (n - 1)]$, and $R_{2,1}$ is the $1 \times 1$ rectangle from $x^{\Id}$ to $[12...(n-2)n(n-1)]$

\item $R_{3,2}, \dots, R_{3,n - 1}$ where $R_{3,i}$ is the $i \times 1$ rectangle from $[(n - i + 1)23\dots (n - i)(n - i + 2)\dots n1]$ to $[(n - i)23\dots (n - i - 1)(n - i + 1)\dots n1]$, and $R_{3,1} = R_{1,1}$.

\item $R_{4,2}, \dots, R_{4,n - 1}$ where $R_{4,i}$ is the $i \times 1$ rectangle from $[12 \dots (n - i)(n - i + 2) \dots n (n - i + 1)]$ to $[12 \dots (n - i - 1)(n - i + 1) \dots n (n - i)]$, and $R_{4,1} = R_{2,1}$.

\item $R_{5, 1}, \dots, R_{5, n-2}$, where $R_{5, i}$ is the $1 \times i$ rectangle from $[12\dots (n - i - 2)(n - i)n(n - i + 1)\dots (n - 1)(n - i - 1)]$ to $[12\dots (n - i - 2)n(n - i)\dots (n - 1)(n - i - 1)]$

\item $R_{6, 1}, \dots, R_{6, n-2}$, where $R_{6, i}$ is the $i \times 1$ rectangle from $[12\dots (n - i - 2)n(n - i - 1)(n - i + 1)\dots (n - 1)(n - i)]$ to $[12\dots (n - i - 2)n(n - i)\dots (n - 1)(n - i - 1)]$

\end{itemize}

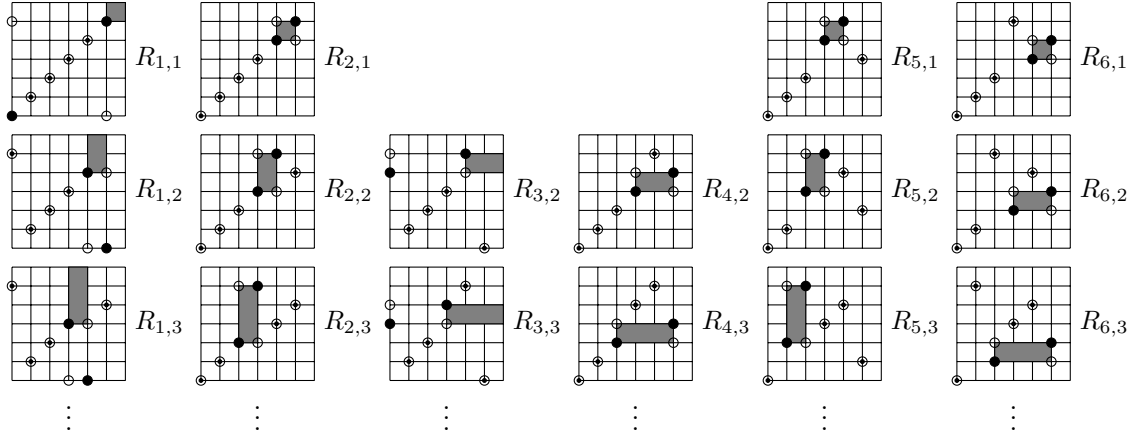
\begin{figure}\begin{tikzpicture}[scale=0.25]

\foreach \i in {0,1,2,3,4,5}{
    \foreach \j in {1,2}{

      \begin{scope}[xshift=10*\i cm,yshift=-7*\j cm] 

        \draw (0,0) grid (6,6);
        
      \end{scope}

}

\node[anchor=north] at (3 + 10*\i, -14) {\vdots};

}

    \begin{scope}[xshift=0 cm,yshift=0 cm] 

        \draw (0,0) grid (6,6);
        
      \end{scope}

    \begin{scope}[xshift=10 cm,yshift=0 cm] 

        \draw (0,0) grid (6,6);
        
      \end{scope}

    \begin{scope}[xshift=40 cm,yshift=0 cm] 

        \draw (0,0) grid (6,6);
        
      \end{scope}
    
    \begin{scope}[xshift=50 cm,yshift=0 cm] 

        \draw (0,0) grid (6,6);
        
      \end{scope}

  \begin{scope}[xshift=0cm, yshift=0 cm]
    \node[anchor=west] at (6,3) {$R_{1,1}$};

    \draw[fill=gray] (5,5) rectangle (6,6);

    \foreach \a/\b in {1/1, 2/2, 3/3, 4/4}{

      \filldraw[black] (\a,\b) circle (0.1);
      \draw (\a,\b) circle (0.25);

    }

    \foreach \a/\b in {0/0, 5/5}{

      \filldraw[black] (\a,\b) circle (0.25);

    }

    \foreach \a/\b in {0/5, 5/0}{

      \draw (\a,\b) circle (0.25);

    }
    
  \end{scope}

  \begin{scope}[xshift=0cm, yshift=-7 cm]
    \node[anchor=west] at (6,3) {$R_{1,2}$};

    \draw[fill=gray] (4,4) rectangle (5,6);

    \foreach \a/\b in {0/5,1/1,2/2,3/3}{

      \filldraw[black] (\a,\b) circle (0.1);
      \draw (\a,\b) circle (0.25);

    }

    \foreach \a/\b in {4/4, 5/0}{

      \filldraw[black] (\a,\b) circle (0.25);

    }

    \foreach \a/\b in {4/0, 5/4}{

      \draw (\a,\b) circle (0.25);

    }
    
  \end{scope}

  \begin{scope}[xshift=0cm, yshift=-14 cm]
    \node[anchor=west] at (6,3) {$R_{1,3}$};

    \draw[fill=gray] (3,3) rectangle (4,6);

    \foreach \a/\b in {0/5, 1/1, 2/2, 5/4}{

      \filldraw[black] (\a,\b) circle (0.1);
      \draw (\a,\b) circle (0.25);

    }

    \foreach \a/\b in {3/3, 4/0}{

      \filldraw[black] (\a,\b) circle (0.25);

    }

    \foreach \a/\b in {3/0, 4/3}{

      \draw (\a,\b) circle (0.25);

    }
    
  \end{scope}

    \begin{scope}[xshift=10cm, yshift=0 cm]
    \node[anchor=west] at (6,3) {$R_{2,1}$};

    \draw[fill=gray] (4,4) rectangle (5,5);

    \foreach \a/\b in {0/0, 1/1, 2/2, 3/3}{

      \filldraw[black] (\a,\b) circle (0.1);
      \draw (\a,\b) circle (0.25);

    }

    \foreach \a/\b in {4/4, 5/5}{

      \filldraw[black] (\a,\b) circle (0.25);

    }

    \foreach \a/\b in {4/5, 5/4}{

      \draw (\a,\b) circle (0.25);

    }
    
  \end{scope}

  \begin{scope}[xshift=10cm, yshift=-7 cm]
    \node[anchor=west] at (6,3) {$R_{2,2}$};

    \draw[fill=gray] (3,3) rectangle (4,5);

    \foreach \a/\b in {0/0,1/1,2/2,5/4}{

      \filldraw[black] (\a,\b) circle (0.1);
      \draw (\a,\b) circle (0.25);

    }

    \foreach \a/\b in {3/3, 4/5}{

      \filldraw[black] (\a,\b) circle (0.25);

    }

    \foreach \a/\b in {3/5, 4/3}{

      \draw (\a,\b) circle (0.25);

    }
    
  \end{scope}

  \begin{scope}[xshift=10cm, yshift=-14 cm]
    \node[anchor=west] at (6,3) {$R_{2,3}$};

    \draw[fill=gray] (2,2) rectangle (3,5);

    \foreach \a/\b in {0/0, 1/1, 4/3, 5/4}{

      \filldraw[black] (\a,\b) circle (0.1);
      \draw (\a,\b) circle (0.25);

    }

    \foreach \a/\b in {2/2, 3/5}{

      \filldraw[black] (\a,\b) circle (0.25);

    }

    \foreach \a/\b in {2/5, 3/2}{

      \draw (\a,\b) circle (0.25);

    }
    
  \end{scope}

  \begin{scope}[xshift=20cm, yshift=-7 cm]
    \node[anchor=west] at (6,3) {$R_{3,2}$};

    \draw[fill=gray] (4,4) rectangle (6,5);

    \foreach \a/\b in {1/1,2/2,3/3,5/0}{

      \filldraw[black] (\a,\b) circle (0.1);
      \draw (\a,\b) circle (0.25);

    }

    \foreach \a/\b in {0/4,4/5}{

      \filldraw[black] (\a,\b) circle (0.25);

    }

    \foreach \a/\b in {0/5,4/4}{

      \draw (\a,\b) circle (0.25);

    }
    
  \end{scope}

  \begin{scope}[xshift=20cm, yshift=-14 cm]
    \node[anchor=west] at (6,3) {$R_{3,3}$};

    \draw[fill=gray] (3,3) rectangle (6,4);

    \foreach \a/\b in {1/1, 2/2, 4/5, 5/0}{

      \filldraw[black] (\a,\b) circle (0.1);
      \draw (\a,\b) circle (0.25);

    }

    \foreach \a/\b in {0/3, 3/4}{

      \filldraw[black] (\a,\b) circle (0.25);

    }

    \foreach \a/\b in {0/4, 3/3}{

      \draw (\a,\b) circle (0.25);

    }
    
  \end{scope}

  \begin{scope}[xshift=30cm, yshift=-7 cm]
    \node[anchor=west] at (6,3) {$R_{4,2}$};

    \draw[fill=gray] (3,3) rectangle (5,4);

    \foreach \a/\b in {0/0,1/1,2/2,4/5}{

      \filldraw[black] (\a,\b) circle (0.1);
      \draw (\a,\b) circle (0.25);

    }

    \foreach \a/\b in {3/3, 5/4}{

      \filldraw[black] (\a,\b) circle (0.25);

    }

    \foreach \a/\b in {3/4, 5/3}{

      \draw (\a,\b) circle (0.25);

    }
    
  \end{scope}

  \begin{scope}[xshift=30cm, yshift=-14 cm]
    \node[anchor=west] at (6,3) {$R_{4,3}$};

    \draw[fill=gray] (2,2) rectangle (5,3);

    \foreach \a/\b in {0/0, 1/1, 3/4, 4/5}{

      \filldraw[black] (\a,\b) circle (0.1);
      \draw (\a,\b) circle (0.25);

    }

    \foreach \a/\b in {2/2, 5/3}{

      \filldraw[black] (\a,\b) circle (0.25);

    }

    \foreach \a/\b in {2/3, 5/2}{

      \draw (\a,\b) circle (0.25);

    }
    
  \end{scope}

  \begin{scope}[xshift=40cm, yshift=0 cm]
    \node[anchor=west] at (6,3) {$R_{5,1}$};

    \draw[fill=gray] (3,4) rectangle (4,5);

    \foreach \a/\b in {0/0, 1/1, 2/2, 5/3}{

      \filldraw[black] (\a,\b) circle (0.1);
      \draw (\a,\b) circle (0.25);

    }

    \foreach \a/\b in {3/4, 4/5}{

      \filldraw[black] (\a,\b) circle (0.25);

    }

    \foreach \a/\b in {4/4, 3/5}{

      \draw (\a,\b) circle (0.25);

    }
    
  \end{scope}

  \begin{scope}[xshift=40cm, yshift=-7 cm]
    \node[anchor=west] at (6,3) {$R_{5,2}$};

    \draw[fill=gray] (2,3) rectangle (3,5);

    \foreach \a/\b in {0/0,1/1,5/2,4/4}{

      \filldraw[black] (\a,\b) circle (0.1);
      \draw (\a,\b) circle (0.25);

    }

    \foreach \a/\b in {2/3, 3/5}{

      \filldraw[black] (\a,\b) circle (0.25);

    }

    \foreach \a/\b in {3/3, 2/5}{

      \draw (\a,\b) circle (0.25);

    }
    
  \end{scope}

  \begin{scope}[xshift=40cm, yshift=-14 cm]
    \node[anchor=west] at (6,3) {$R_{5,3}$};

    \draw[fill=gray] (1,2) rectangle (2,5);

    \foreach \a/\b in {0/0, 5/1, 3/3, 4/4}{

      \filldraw[black] (\a,\b) circle (0.1);
      \draw (\a,\b) circle (0.25);

    }

    \foreach \a/\b in {1/2, 2/5}{

      \filldraw[black] (\a,\b) circle (0.25);

    }

    \foreach \a/\b in {2/2, 1/5}{

      \draw (\a,\b) circle (0.25);

    }
    
  \end{scope}

  \begin{scope}[xshift=50cm, yshift=0 cm]
    \node[anchor=west] at (6,3) {$R_{6,1}$};

    \draw[fill=gray] (4,3) rectangle (5,4);

    \foreach \a/\b in {0/0, 1/1, 2/2, 3/5}{

      \filldraw[black] (\a,\b) circle (0.1);
      \draw (\a,\b) circle (0.25);

    }

    \foreach \a/\b in {4/3, 5/4}{

      \filldraw[black] (\a,\b) circle (0.25);

    }

    \foreach \a/\b in {4/4, 5/3}{

      \draw (\a,\b) circle (0.25);

    }
    
  \end{scope}

  \begin{scope}[xshift=50cm, yshift=-7 cm]
    \node[anchor=west] at (6,3) {$R_{6,2}$};

    \draw[fill=gray] (3,2) rectangle (5,3);

    \foreach \a/\b in {0/0,1/1,2/5,4/4}{

      \filldraw[black] (\a,\b) circle (0.1);
      \draw (\a,\b) circle (0.25);

    }

    \foreach \a/\b in {3/2, 5/3}{

      \filldraw[black] (\a,\b) circle (0.25);

    }

    \foreach \a/\b in {3/3, 5/2}{

      \draw (\a,\b) circle (0.25);

    }
    
  \end{scope}

  \begin{scope}[xshift=50cm, yshift=-14 cm]
    \node[anchor=west] at (6,3) {$R_{6,3}$};

    \draw[fill=gray] (2,1) rectangle (5,2);

    \foreach \a/\b in {0/0, 1/5, 3/3, 4/4}{

      \filldraw[black] (\a,\b) circle (0.1);
      \draw (\a,\b) circle (0.25);

    }

    \foreach \a/\b in {2/1, 5/2}{

      \filldraw[black] (\a,\b) circle (0.25);

    }

    \foreach \a/\b in {2/2, 5/1}{

      \draw (\a,\b) circle (0.25);

    }
    
  \end{scope}

\end{tikzpicture}\caption{The rectangles $R_{1,i}, R_{2,i}, R_{3,i}, R_{4,i}, R_{5, i}, R_{6, i}$, where each rectangle is drawn from a generator $x$ (drawn by \textbullet) to a generator $y$ (drawn by $\circ$.)}\label{fig5}\end{figure}

(See Figure $\ref{fig5}$.) We cancel each of these rectangles in the boundary as follows: 

\begin{itemize}

\item $R_{1, 1}$ occurs in $\partial U$ twice, from $\partial A_0$ and $\partial B_0$, so it cancels in $\partial U$. For $i = 2, \dots n - 1$, $R_{1, i}$ occurs in $\partial A_{i - 1}$ and $\partial C_{i - 1}$, so they also cancel in $\partial U$.

\item $R_{2, 1}$ occurs in $\partial U$ twice, from $\partial C_1$ and $\partial D_1$, so it cancels in $\partial U$, $R_{2, n - 1}$ occurs in $\partial A_{n - 1}$ and $\partial E_{n - 2}$, and for $i = 2, \dots n - 2$, $R_{2, i}$ occurs in $\partial C_{i}$ and $\partial E_{i}$, so they also cancel in $\partial U$.

\item For $i = 2, \dots n - 1$, $R_{3, i}$ occurs in $\partial B_{i - 1}$ and $\partial D_{i - 1}$.

\item For $i = 2, \dots n - 2$, $R_{4, i}$ occurs in $\partial D_{i}$ and $\partial E_{i - 1}$.

\item $R_{5, 1}$ occurs in $\partial E_{1}$ and $\partial F_{1, 1}$, and for $i = 2, \dots n - 2$, $R_{5, i}$ occurs in $\partial E_{i}$ and $\partial F_{1, i - 1}$.

\item $R_{6, 1}$ occurs in $\partial E_{1}$ and $\partial G_{1, 1}$, and for $i = 2, \dots n - 2$, $R_{6, i}$ occurs in $\partial E_{i}$ and $\partial G_{1, i - 1}$.

\end{itemize}

Next, we consider the following rectangles: 

\begin{itemize}

\item $P_{i, 1} \dots P_{i, n-i-1}$ for each $i = 2 \dots n - 2$, where $P_{i, j}$ is the $1 \times j$ rectangle from $[12\dots (n-i-j-1)(n-i-j+1)(n-i+1)(n-i-j+2)\dots (n-i) (n-i+2) \dots n (n-i-j)]$ to $[12\dots (n-i-j-1)(n-i+1)(n-i-j+1)\dots (n-i) (n-i+2) \dots n (n-i-j)]$, and $P_{1, j} = R_{5, j}$ for each $j = 1, \dots n - 2$.

\item $Q_{i, 1} \dots Q_{i, n-i-1}$ for each $i = 2 \dots n - 2$, where $Q_{i, j}$ is the $j \times 1$ rectangle from $[12\dots (n-i-j-1)n(n-i-j)(n-i-j+2)\dots (n-i)(n-i-j+1)(n-i+1) \dots (n-1)]$ to $[12\dots (n-i-j-1)n(n-i-j+1)\dots (n-i)(n-i-j)(n-i+1) \dots (n-1)]$, and $Q_{1, j} = R_{6, j}$ for each $j = 1, \dots n - 2$.

\end{itemize}

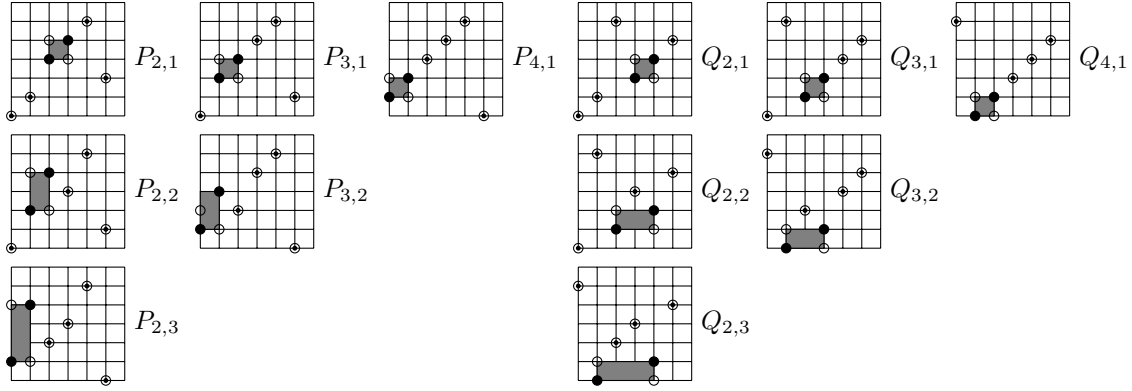
\begin{figure}\begin{tikzpicture}[scale=0.25]

  \begin{scope}[xshift=0cm,yshift=0cm] 

        \draw (0, 0) grid (6, 6);
        \node[anchor=west] at (6, 3){$P_{2, 1}$};
        
      \end{scope}

\begin{scope}[xshift=0cm,yshift=-7cm] 

       \draw (0, 0) grid (6, 6);
       \node[anchor=west] at (6, 3){$P_{2, 2}$};
        
      \end{scope}

\begin{scope}[xshift=0cm,yshift=-14cm] 

      \draw (0, 0) grid (6, 6);
      \node[anchor=west] at (6, 3){$P_{2, 3}$};
        
      \end{scope}

 \begin{scope}[xshift=0cm, yshift=0cm]

    \draw[fill=gray] (2, 3) rectangle (3, 4);

    \foreach \a/\b in {0/0, 1/1, 4/5, 5/2}{

      \filldraw[black] (\a,\b) circle (0.1);
      \draw (\a,\b) circle (0.25);

    }

    \foreach \a/\b in {2/3, 3/4}{

      \filldraw[black] (\a,\b) circle (0.25);

    }

    \foreach \a/\b in {2/4, 3/3}{

      \draw (\a,\b) circle (0.25);

    }
    
  \end{scope}

  \begin{scope}[xshift=0cm, yshift=-7cm]

    \draw[fill=gray] (1, 2) rectangle (2, 4);

    \foreach \a/\b in {0/0, 3/3, 4/5, 5/1}{

      \filldraw[black] (\a,\b) circle (0.1);
      \draw (\a,\b) circle (0.25);

    }

    \foreach \a/\b in {1/2, 2/4}{

      \filldraw[black] (\a,\b) circle (0.25);

    }

    \foreach \a/\b in {1/4, 2/2}{

      \draw (\a,\b) circle (0.25);

    }
    
  \end{scope}

  \begin{scope}[xshift=0cm, yshift=-14cm]

    \draw[fill=gray] (0, 1) rectangle (1, 4);

    \foreach \a/\b in {2/2, 3/3, 4/5, 5/0}{

      \filldraw[black] (\a,\b) circle (0.1);
      \draw (\a,\b) circle (0.25);

    }

    \foreach \a/\b in {0/1, 1/4}{

      \filldraw[black] (\a,\b) circle (0.25);

    }

    \foreach \a/\b in {0/4, 1/1}{

      \draw (\a,\b) circle (0.25);

    }
    
  \end{scope}

  \begin{scope}[xshift=10cm,yshift=0cm] 

        \draw (0, 0) grid (6, 6);
        \node[anchor=west] at (6, 3){$P_{3, 1}$};
        
      \end{scope}

\begin{scope}[xshift=10cm,yshift=-7cm] 

       \draw (0, 0) grid (6, 6);
       \node[anchor=west] at (6, 3){$P_{3, 2}$};
        
      \end{scope}

 \begin{scope}[xshift=10cm, yshift=0cm]

    \draw[fill=gray] (1, 2) rectangle (2, 3);

    \foreach \a/\b in {0/0, 3/4, 4/5, 5/1}{

      \filldraw[black] (\a,\b) circle (0.1);
      \draw (\a,\b) circle (0.25);

    }

    \foreach \a/\b in {1/2, 2/3}{

      \filldraw[black] (\a,\b) circle (0.25);

    }

    \foreach \a/\b in {1/3, 2/2}{

      \draw (\a,\b) circle (0.25);

    }
    
  \end{scope}

  \begin{scope}[xshift=10cm, yshift=-7cm]

    \draw[fill=gray] (0, 1) rectangle (1, 3);

    \foreach \a/\b in {2/2, 3/4, 4/5, 5/0}{

      \filldraw[black] (\a,\b) circle (0.1);
      \draw (\a,\b) circle (0.25);

    }

    \foreach \a/\b in {0/1, 1/3}{

      \filldraw[black] (\a,\b) circle (0.25);

    }

    \foreach \a/\b in {0/2, 1/1}{

      \draw (\a,\b) circle (0.25);

    }
    
  \end{scope}

  \begin{scope}[xshift=20cm,yshift=0cm] 

        \draw (0, 0) grid (6, 6);
        \node[anchor=west] at (6, 3){$P_{4, 1}$};
        
      \end{scope}

 \begin{scope}[xshift=20cm, yshift=0cm]

    \draw[fill=gray] (0, 1) rectangle (1, 2);

    \foreach \a/\b in {2/3, 3/4, 4/5, 5/0}{

      \filldraw[black] (\a,\b) circle (0.1);
      \draw (\a,\b) circle (0.25);

    }

    \foreach \a/\b in {0/1, 1/2}{

      \filldraw[black] (\a,\b) circle (0.25);

    }

    \foreach \a/\b in {0/2, 1/1}{

      \draw (\a,\b) circle (0.25);

    }
    
  \end{scope}

\begin{scope}[xshift=30cm,yshift=0cm] 

        \draw (0, 0) grid (6, 6);
        \node[anchor=west] at (6, 3){$Q_{2, 1}$};
        
      \end{scope}

\begin{scope}[xshift=30cm,yshift=-7cm] 

       \draw (0, 0) grid (6, 6);
       \node[anchor=west] at (6, 3){$Q_{2, 2}$};
        
      \end{scope}

\begin{scope}[xshift=30cm,yshift=-14cm] 

      \draw (0, 0) grid (6, 6);
      \node[anchor=west] at (6, 3){$Q_{2, 3}$};
        
      \end{scope}

 \begin{scope}[xshift=30cm, yshift=0cm]

    \draw[fill=gray] (3, 2) rectangle (4, 3);

    \foreach \a/\b in {0/0, 1/1, 5/4, 2/5}{

      \filldraw[black] (\a,\b) circle (0.1);
      \draw (\a,\b) circle (0.25);

    }

    \foreach \a/\b in {3/2, 4/3}{

      \filldraw[black] (\a,\b) circle (0.25);

    }

    \foreach \a/\b in {4/2, 3/3}{

      \draw (\a,\b) circle (0.25);

    }
    
  \end{scope}

  \begin{scope}[xshift=30cm, yshift=-7cm]

    \draw[fill=gray] (2, 1) rectangle (4, 2);

    \foreach \a/\b in {0/0, 3/3, 5/4, 1/5}{

      \filldraw[black] (\a,\b) circle (0.1);
      \draw (\a,\b) circle (0.25);

    }

    \foreach \a/\b in {2/1, 4/2}{

      \filldraw[black] (\a,\b) circle (0.25);

    }

    \foreach \a/\b in {4/1, 2/2}{

      \draw (\a,\b) circle (0.25);

    }
    
  \end{scope}

  \begin{scope}[xshift=30cm, yshift=-14cm]

    \draw[fill=gray] (1, 0) rectangle (4, 1);

    \foreach \a/\b in {2/2, 3/3, 5/4, 0/5}{

      \filldraw[black] (\a,\b) circle (0.1);
      \draw (\a,\b) circle (0.25);

    }

    \foreach \a/\b in {1/0, 4/1}{

      \filldraw[black] (\a,\b) circle (0.25);

    }

    \foreach \a/\b in {4/0, 1/1}{

      \draw (\a,\b) circle (0.25);

    }
    
  \end{scope}

  \begin{scope}[xshift=40cm,yshift=0cm] 

        \draw (0, 0) grid (6, 6);
        \node[anchor=west] at (6, 3){$Q_{3, 1}$};
        
      \end{scope}

\begin{scope}[xshift=40cm,yshift=-7cm] 

       \draw (0, 0) grid (6, 6);
       \node[anchor=west] at (6, 3){$Q_{3, 2}$};
        
      \end{scope}

 \begin{scope}[xshift=40cm, yshift=0cm]

    \draw[fill=gray] (2, 1) rectangle (3, 2);

    \foreach \a/\b in {0/0, 4/3, 5/4, 1/5}{

      \filldraw[black] (\a,\b) circle (0.1);
      \draw (\a,\b) circle (0.25);

    }

    \foreach \a/\b in {2/1, 3/2}{

      \filldraw[black] (\a,\b) circle (0.25);

    }

    \foreach \a/\b in {3/1, 2/2}{

      \draw (\a,\b) circle (0.25);

    }
    
  \end{scope}

  \begin{scope}[xshift=40cm, yshift=-7cm]

    \draw[fill=gray] (1, 0) rectangle (3, 1);

    \foreach \a/\b in {2/2, 4/3, 5/4, 0/5}{

      \filldraw[black] (\a,\b) circle (0.1);
      \draw (\a,\b) circle (0.25);

    }

    \foreach \a/\b in {1/0, 3/1}{

      \filldraw[black] (\a,\b) circle (0.25);

    }

    \foreach \a/\b in {3/0, 1/1}{

      \draw (\a,\b) circle (0.25);

    }
    
  \end{scope}

  \begin{scope}[xshift=50cm,yshift=0cm] 

        \draw (0, 0) grid (6, 6);
        \node[anchor=west] at (6, 3){$Q_{4, 1}$};
        
      \end{scope}

 \begin{scope}[xshift=50cm, yshift=0cm]

    \draw[fill=gray] (1, 0) rectangle (2, 1);

    \foreach \a/\b in {3/2, 4/3, 5/4, 0/5}{

      \filldraw[black] (\a,\b) circle (0.1);
      \draw (\a,\b) circle (0.25);

    }

    \foreach \a/\b in {1/0, 2/1}{

      \filldraw[black] (\a,\b) circle (0.25);

    }

    \foreach \a/\b in {2/0, 1/1}{

      \draw (\a,\b) circle (0.25);

    }
    
  \end{scope}

\end{tikzpicture}\caption{The rectangles $P_{i, j}$ and $Q_{i, j}$ in the special case of a $6 \times 6$ grid, where each domain is drawn from a generator $x$ (drawn by \textbullet) to a generator $y$ (drawn by $\circ$)}. \label{fig4}\end{figure}

(see Figure $\ref{fig4}$) We cancel each of these rectangles in the boundary as follows: 

\begin{itemize}

\item $P_{n-2, 1}$ occurs in $\partial F_{n-3, 1}$ and $\partial B_{n - 2}$. For $i = 2, \dots n - 3$, $P_{i, j}$ occurs in $\partial F_{i - 1, j}$ and either $\partial F_{i, j - 1}$ if $j \geq 2$ or $\partial F_{i, 1}$ if $j = 1$.

\item $Q_{n-2, 1}$ occurs in $\partial G_{n-3, 1}$ and $\partial A_{n - 2}$. For $i = 2, \dots n - 3$, $Q_{i, j}$ occurs in $\partial G_{i - 1, j}$ and either $\partial G_{i, j - 1}$ if $j \geq 2$ or $\partial G_{i, 1}$ if $j = 1$.

\end{itemize}

Finally, the remaining rectangles have the following form:
\begin{itemize}

\item $R_{1, 1}' \dots R_{1, n-1}'$, where $R_{1, i}'$ is the $(n - i) \times 1$ rectangle from $[n23\dots (n-i)1(n-i+1) \dots (n-1)]$ to $[12\dots (n-i)n(n-i+1) \dots (n-1)]$.

\item $R_{2, 1}' \dots R_{2, n-1}'$, where $R_{2, i}'$ is the $1 \times (n - i)$ rectangle from $[(n-i+1)23\dots (n-i)(n-i+2)\dots n1]$ to $[12\dots (n-i)(n-i+2) \dots n(n-i+1)]$.

\item $P_{i, 2}' \dots P_{i, n-i-1}'$ for $i = 1 \dots n - 3$, where $P_{i, j}'$ is the $j \times 1$ rectangle from $[12\dots(n-i-j-1)(n-i)(n-i-j+1) \dots (n-i-1)(n-i+1) \dots n(n-i-j)]$ to $[12\dots(n-i-j-1)(n-i + 1)(n-i-j+1) \dots (n-i)(n-i+2) \dots n(n-i-j)]$, and $P_{i, 1}' = P_{i, 1}$.

\item $Q_{i, 2}' \dots Q_{i, n-i-1}'$ for $i = 1 \dots n - 3$, where $Q_{i, j}'$ is the $1 \times j$ rectangle from $[12\dots (n-i-j-1)n(n-i-j+1)\dots (n-i-1)(n-i-j)(n-i)\dots(n-1)]$ to $[12\dots (n-i-j-1)n(n-i-j+1)\dots (n-i)(n-i-j)(n-i+1)\dots(n-1)]$, and $Q_{i, 1}' = Q_{i, 1}$.

\end{itemize}

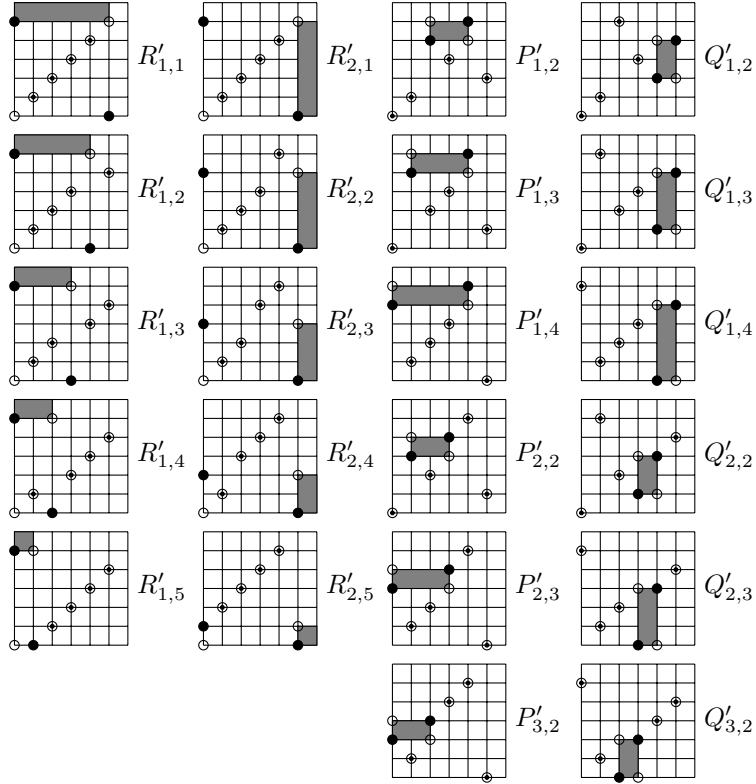
\begin{figure}\begin{tikzpicture}[scale=0.25]

\foreach \i in {1,2,3,4,5}{

      \begin{scope}[xshift=0 cm,yshift=-7*\i cm] 

        \draw (0,0) grid (6,6);
        \node[anchor=west] at (6, 3) {$R_{1, \i}'$};
        
      \end{scope}


}

\begin{scope}[xshift=0cm, yshift=-7 cm]

    \draw[fill=gray] (0,5) rectangle (5,6);

    \foreach \a/\b in {1/1,2/2,3/3,4/4}{

      \filldraw[black] (\a,\b) circle (0.1);
      \draw (\a,\b) circle (0.25);

    }

    \foreach \a/\b in {0/5, 5/0}{

      \filldraw[black] (\a,\b) circle (0.25);

    }

    \foreach \a/\b in {0/0, 5/5}{

      \draw (\a,\b) circle (0.25);

    }
    
  \end{scope}

  \begin{scope}[xshift=0cm, yshift=-14 cm]

    \draw[fill=gray] (0,5) rectangle (4,6);

    \foreach \a/\b in {1/1,2/2,3/3,5/4}{

      \filldraw[black] (\a,\b) circle (0.1);
      \draw (\a,\b) circle (0.25);

    }

    \foreach \a/\b in {0/5, 4/0}{

      \filldraw[black] (\a,\b) circle (0.25);

    }

    \foreach \a/\b in {0/0, 4/5}{

      \draw (\a,\b) circle (0.25);

    }
    
  \end{scope}

  \begin{scope}[xshift=0cm, yshift=-21 cm]

    \draw[fill=gray] (0,5) rectangle (3,6);

    \foreach \a/\b in {1/1,2/2,4/3,5/4}{

      \filldraw[black] (\a,\b) circle (0.1);
      \draw (\a,\b) circle (0.25);

    }

    \foreach \a/\b in {0/5, 3/0}{

      \filldraw[black] (\a,\b) circle (0.25);

    }

    \foreach \a/\b in {0/0, 3/5}{

      \draw (\a,\b) circle (0.25);

    }

    \end{scope}

    \begin{scope}[xshift=0cm, yshift=-28 cm]

    \draw[fill=gray] (0,5) rectangle (2,6);

    \foreach \a/\b in {1/1,3/2,4/3,5/4}{

      \filldraw[black] (\a,\b) circle (0.1);
      \draw (\a,\b) circle (0.25);

    }

    \foreach \a/\b in {0/5, 2/0}{

      \filldraw[black] (\a,\b) circle (0.25);

    }

    \foreach \a/\b in {0/0, 2/5}{

      \draw (\a,\b) circle (0.25);

    }
    
  \end{scope}

  \begin{scope}[xshift=0cm, yshift=-35 cm]

    \draw[fill=gray] (0,5) rectangle (1,6);

    \foreach \a/\b in {2/1,3/2,4/3,5/4}{

      \filldraw[black] (\a,\b) circle (0.1);
      \draw (\a,\b) circle (0.25);

    }

    \foreach \a/\b in {0/5, 1/0}{

      \filldraw[black] (\a,\b) circle (0.25);

    }

    \foreach \a/\b in {0/0, 1/5}{

      \draw (\a,\b) circle (0.25);

    }
    
  \end{scope}

\foreach \i in {1,2,3,4,5}{

      \begin{scope}[xshift=10 cm,yshift=-7*\i cm] 

        \draw (0,0) grid (6,6);
        \node[anchor=west] at (6, 3) {$R_{2, \i}'$};
      \end{scope}

}

\begin{scope}[xshift=10cm, yshift=-7 cm]

    \draw[fill=gray] (5,0) rectangle (6,5);

    \foreach \a/\b in {1/1,2/2,3/3,4/4}{

      \filldraw[black] (\a,\b) circle (0.1);
      \draw (\a,\b) circle (0.25);

    }

    \foreach \a/\b in {0/5, 5/0}{

      \filldraw[black] (\a,\b) circle (0.25);

    }

    \foreach \a/\b in {0/0, 5/5}{

      \draw (\a,\b) circle (0.25);

    }
    
  \end{scope}

  \begin{scope}[xshift=10cm, yshift=-14 cm]

    \draw[fill=gray] (5,0) rectangle (6,4);

    \foreach \a/\b in {1/1,2/2,3/3,4/5}{

      \filldraw[black] (\a,\b) circle (0.1);
      \draw (\a,\b) circle (0.25);

    }

    \foreach \a/\b in {0/4, 5/0}{

      \filldraw[black] (\a,\b) circle (0.25);

    }

    \foreach \a/\b in {0/0, 5/4}{

      \draw (\a,\b) circle (0.25);

    }
    
  \end{scope}

  \begin{scope}[xshift=10cm, yshift=-21 cm]

    \draw[fill=gray] (5,0) rectangle (6,3);

    \foreach \a/\b in {1/1,2/2,3/4,4/5}{

      \filldraw[black] (\a,\b) circle (0.1);
      \draw (\a,\b) circle (0.25);

    }

    \foreach \a/\b in {0/3, 5/0}{

      \filldraw[black] (\a,\b) circle (0.25);

    }

    \foreach \a/\b in {0/0, 5/3}{

      \draw (\a,\b) circle (0.25);

    }

    \end{scope}

    \begin{scope}[xshift=10cm, yshift=-28 cm]

    \draw[fill=gray] (5,0) rectangle (6,2);

    \foreach \a/\b in {1/1,2/3,3/4,4/5}{

      \filldraw[black] (\a,\b) circle (0.1);
      \draw (\a,\b) circle (0.25);

    }

    \foreach \a/\b in {0/2, 5/0}{

      \filldraw[black] (\a,\b) circle (0.25);

    }

    \foreach \a/\b in {0/0, 5/2}{

      \draw (\a,\b) circle (0.25);

    }
    
  \end{scope}

  \begin{scope}[xshift=10cm, yshift=-35 cm]

    \draw[fill=gray] (5,0) rectangle (6,1);

    \foreach \a/\b in {1/2,2/3,3/4,4/5}{

      \filldraw[black] (\a,\b) circle (0.1);
      \draw (\a,\b) circle (0.25);

    }

    \foreach \a/\b in {0/1, 5/0}{

      \filldraw[black] (\a,\b) circle (0.25);

    }

    \foreach \a/\b in {0/0, 5/1}{

      \draw (\a,\b) circle (0.25);

    }
    
  \end{scope}

  \foreach \i in {1,2,3,4,5,6}{

      \begin{scope}[xshift=20 cm,yshift=-7*\i cm] 

        \draw (0,0) grid (6,6);
        
      \end{scope}


}

\begin{scope}[xshift=20cm, yshift=-7 cm]

    \node[anchor=west] at (6, 3) {$P_{1, 2}'$};
    \draw[fill=gray] (2,4) rectangle (4,5);

    \foreach \a/\b in {0/0,1/1,3/3,5/2}{

      \filldraw[black] (\a,\b) circle (0.1);
      \draw (\a,\b) circle (0.25);

    }

    \foreach \a/\b in {2/4, 4/5}{

      \filldraw[black] (\a,\b) circle (0.25);

    }

    \foreach \a/\b in {2/5, 4/4}{

      \draw (\a,\b) circle (0.25);

    }
    
  \end{scope}

\begin{scope}[xshift=20cm, yshift=-14 cm]

    \node[anchor=west] at (6, 3) {$P_{1, 3}'$};
    \draw[fill=gray] (1,4) rectangle (4,5);

    \foreach \a/\b in {0/0,2/2,3/3,5/1}{

      \filldraw[black] (\a,\b) circle (0.1);
      \draw (\a,\b) circle (0.25);

    }

    \foreach \a/\b in {1/4, 4/5}{

      \filldraw[black] (\a,\b) circle (0.25);

    }

    \foreach \a/\b in {1/5, 4/4}{

      \draw (\a,\b) circle (0.25);

    }
    
  \end{scope}

\begin{scope}[xshift=20cm, yshift=-21 cm]

    \node[anchor=west] at (6, 3) {$P_{1, 4}'$};
    \draw[fill=gray] (0,4) rectangle (4,5);

    \foreach \a/\b in {1/1,2/2,3/3,5/0}{

      \filldraw[black] (\a,\b) circle (0.1);
      \draw (\a,\b) circle (0.25);

    }

    \foreach \a/\b in {0/4, 4/5}{

      \filldraw[black] (\a,\b) circle (0.25);

    }

    \foreach \a/\b in {0/5, 4/4}{

      \draw (\a,\b) circle (0.25);

    }
    
  \end{scope}

  \begin{scope}[xshift=20cm, yshift=-28 cm]

    \node[anchor=west] at (6, 3) {$P_{2, 2}'$};
    \draw[fill=gray] (1,3) rectangle (3,4);

    \foreach \a/\b in {0/0,2/2,4/5,5/1}{

      \filldraw[black] (\a,\b) circle (0.1);
      \draw (\a,\b) circle (0.25);

    }

    \foreach \a/\b in {1/3, 3/4}{

      \filldraw[black] (\a,\b) circle (0.25);

    }

    \foreach \a/\b in {1/4, 3/3}{

      \draw (\a,\b) circle (0.25);

    }
    
  \end{scope}

  \begin{scope}[xshift=20cm, yshift=-35 cm]

    \node[anchor=west] at (6, 3) {$P_{2, 3}'$};
    \draw[fill=gray] (0,3) rectangle (3,4);

    \foreach \a/\b in {1/1,2/2,4/5,5/0}{

      \filldraw[black] (\a,\b) circle (0.1);
      \draw (\a,\b) circle (0.25);

    }

    \foreach \a/\b in {0/3, 3/4}{

      \filldraw[black] (\a,\b) circle (0.25);

    }

    \foreach \a/\b in {0/4, 3/3}{

      \draw (\a,\b) circle (0.25);

    }
    
  \end{scope}

  \begin{scope}[xshift=20cm, yshift=-42 cm]

    \node[anchor=west] at (6, 3) {$P_{3, 2}'$};
    \draw[fill=gray] (0,2) rectangle (2,3);

    \foreach \a/\b in {1/1,3/4,4/5,5/0}{

      \filldraw[black] (\a,\b) circle (0.1);
      \draw (\a,\b) circle (0.25);

    }

    \foreach \a/\b in {0/2, 2/3}{

      \filldraw[black] (\a,\b) circle (0.25);

    }

    \foreach \a/\b in {0/3, 2/2}{

      \draw (\a,\b) circle (0.25);

    }

    \end{scope}

\foreach \i in {1,2,3,4,5,6}{

      \begin{scope}[xshift=30 cm,yshift=-7*\i cm] 

        \draw (0,0) grid (6,6);
        
      \end{scope}


}

\begin{scope}[xshift=30cm, yshift=-7 cm]

    \node[anchor=west] at (6, 3) {$Q_{1, 2}'$};
    \draw[fill=gray] (4,2) rectangle (5,4);

    \foreach \a/\b in {0/0,1/1,3/3,2/5}{

      \filldraw[black] (\a,\b) circle (0.1);
      \draw (\a,\b) circle (0.25);

    }

    \foreach \a/\b in {4/2, 5/4}{

      \filldraw[black] (\a,\b) circle (0.25);

    }

    \foreach \a/\b in {5/2, 4/4}{

      \draw (\a,\b) circle (0.25);

    }
    
  \end{scope}

\begin{scope}[xshift=30cm, yshift=-14 cm]

    \node[anchor=west] at (6, 3) {$Q_{1, 3}'$};
    \draw[fill=gray] (4,1) rectangle (5,4);

    \foreach \a/\b in {0/0,2/2,3/3,1/5}{

      \filldraw[black] (\a,\b) circle (0.1);
      \draw (\a,\b) circle (0.25);

    }

    \foreach \a/\b in {4/1, 5/4}{

      \filldraw[black] (\a,\b) circle (0.25);

    }

    \foreach \a/\b in {5/1, 4/4}{

      \draw (\a,\b) circle (0.25);

    }
    
  \end{scope}

\begin{scope}[xshift=30cm, yshift=-21 cm]

    \node[anchor=west] at (6, 3) {$Q_{1, 4}'$};
    \draw[fill=gray] (4,0) rectangle (5,4);

    \foreach \a/\b in {1/1,2/2,3/3,0/5}{

      \filldraw[black] (\a,\b) circle (0.1);
      \draw (\a,\b) circle (0.25);

    }

    \foreach \a/\b in {4/0, 5/4}{

      \filldraw[black] (\a,\b) circle (0.25);

    }

    \foreach \a/\b in {5/0, 4/4}{

      \draw (\a,\b) circle (0.25);

    }
    
  \end{scope}

  \begin{scope}[xshift=30cm, yshift=-28 cm]

    \node[anchor=west] at (6, 3) {$Q_{2, 2}'$};
    \draw[fill=gray] (3,1) rectangle (4,3);

    \foreach \a/\b in {0/0,2/2,5/4,1/5}{

      \filldraw[black] (\a,\b) circle (0.1);
      \draw (\a,\b) circle (0.25);

    }

    \foreach \a/\b in {3/1, 4/3}{

      \filldraw[black] (\a,\b) circle (0.25);

    }

    \foreach \a/\b in {4/1, 3/3}{

      \draw (\a,\b) circle (0.25);

    }
    
  \end{scope}

  \begin{scope}[xshift=30cm, yshift=-35 cm]

    \node[anchor=west] at (6, 3) {$Q_{2, 3}'$};
    \draw[fill=gray] (3,0) rectangle (4,3);

    \foreach \a/\b in {1/1,2/2,5/4,0/5}{

      \filldraw[black] (\a,\b) circle (0.1);
      \draw (\a,\b) circle (0.25);

    }

    \foreach \a/\b in {3/0, 4/3}{

      \filldraw[black] (\a,\b) circle (0.25);

    }

    \foreach \a/\b in {4/0, 3/3}{

      \draw (\a,\b) circle (0.25);

    }
    
  \end{scope}

  \begin{scope}[xshift=30cm, yshift=-42 cm]

    \node[anchor=west] at (6, 3) {$Q_{3, 2}'$};
    \draw[fill=gray] (2,0) rectangle (3,2);

    \foreach \a/\b in {1/1,4/3,5/4,0/5}{

      \filldraw[black] (\a,\b) circle (0.1);
      \draw (\a,\b) circle (0.25);

    }

    \foreach \a/\b in {2/0, 3/2}{

      \filldraw[black] (\a,\b) circle (0.25);

    }

    \foreach \a/\b in {3/0, 2/2}{

      \draw (\a,\b) circle (0.25);

    }

    \end{scope}

\end{tikzpicture}

\caption{The rectangles $R_{1, i}', R_{2, i}', P_{i, j}', Q_{i, j}'$ in the special case of a $6 \times 6$ grid, where each domain is drawn from a generator $x$ (drawn by \textbullet) to a generator $y$ (drawn by $\circ$)}.\label{fig6}\end{figure}

(See Figure $\ref{fig6}$.) We cancel each of these rectangles in the boundary as follows: 

\begin{itemize}

\item $R_{1, 1}'$ occurs in $\partial B_0$ and $\partial C_1$, $R_{1, n-1}'$ occurs in $\partial A_{n-1}$ and $\partial C_{n-2}$, and for $i = 2, \dots, n-2$, $R_{1, i}'$ occurs in $\partial C_{i - 1}$ and $\partial C_{i}$.

\item $R_{2, 1}'$ occurs in $\partial A_0$ and $\partial D_1$, $R_{2, n-1}'$ occurs in $\partial B_{n-1}$ and $\partial D_{n-2}$, and for $i = 2, \dots, n-2$, $R_{2, i}'$ occurs in $\partial D_{i - 1}$ and $\partial D_{i}$.

\item $P_{i, n-i-1}'$ occurs in $\partial F_{i, n-i-2}$ and $\partial B_{i}$. For $2 \leq j \leq n - i - 2$, $P_{i, j}'$ occurs in $\partial F_{i, j-1}$ and $\partial F_{i, j}$.

\item $Q_{i, n-i-1}'$ occurs in $\partial G_{i, n-i-2}$ and $\partial A_i$. For $2 \leq j \leq n - i - 2$, $Q_{i, j}'$ occurs in $\partial G_{i, j-1}$ and $\partial G_{i, j}$.

\end{itemize}

Since these are the only rectangles produced by $\partial A_i, \partial B_i, \partial C_i, \partial D_i, \partial E_i, \partial F_{i, j}, \partial G_{i, j}$, we conclude that indeed $\partial U = 0$.\end{proof}

\begin{lemma}$U$ is not homologous to zero in $\CD_*$.\label{lem8}\end{lemma}

\begin{proof} Let $r$ be the $2$-cochain which is $1$ on the rightmost vertical annulus from any generator to itself, and zero on all other domains; we will first show that $r$ is a cocycle, at which point it suffices to show that $r(U) \neq 0$. Let $E$ be an index $3$ domain. If $E$ does not contain the rightmost vertical annulus, then clearly $\delta r(E) = 0$. If $E$ does contain the rightmost vertical annulus, then $E$ can be written exactly two ways as the product of the rightmost vertical annulus $V_{(n)}$ with an index $1$ domain: $E = D * V_{(n)} = V_{(n)} * D$. So $\delta r(E) = 0$ and therefore $r$ is a cocycle, and $r(U) = 1$ since $U$ contains exactly one copy of the rightmost vertical annulus.\end{proof}

\begin{proof}[Proof of Proposition \ref{prop2}]This immediately follows from Lemmas $\ref{lem7}$ and $\ref{lem8}$ and Proposition $\ref{prop1}$.\end{proof}

\section{Sign Assigments}

In order to extend $\CD_*$ over $\Z$ coefficients (and to frame some of the $0$-dimensional moduli spaces in the Manolescu-Sarkar construction), we need a sign assignment for $\CD_*$, which is a particular $\Z/2$-valued $1$-cochain on $\CD_*$. The following conditions for a sign assignment ensures that $1$-dimensional moduli spaces are frameable, since their boundaries must have opposite signs---see \cite{MS} for more details, and note also that this agrees with the sign assignments defined by \cite{MOST} and \cite{GA}, though we are giving a new proof of their existence.

\begin{definition}A sign assigment for $\G$ is a $\Z/2$-valued $1$-cochain $s$ on $\CD_*$ such that
\begin{enumerate}

\item (Square Rule) If $D_1, D_2, D_3, D_4$ are distinct rectangles such that $D_1 * D_2 = D_3 * D_4 = E$ which is not an annulus, then $s(D_1) + s(D_2) = s(D_3) + s(D_4) + 1$.


\item (Annuli) If $D_1, D_2$ are rectangles such that $D_1 * D_2$ is a vertical annulus, $s(D_1) = s(D_2) + 1$. If $D_1, D_2$ are rectangles such that $D_1 * D_2$ is a horizontal annulus, $s(D_1) = s(D_2)$.

\end{enumerate}\label{def2}\end{definition}

In order to prove that such a sign assignment exists, we will show that the $2$-cocycle that we hypothesize to be $\delta s$ is indeed a $2$-coboundary.

\begin{lemma}Let $T$ be the $2$-cochain with values
\begin{enumerate}

\item (Square Rule) For any index $2$ domain $D$ that is not an annulus, $T(D) = 1$.


\item (Annuli) $T(V) = 1$ for all vertical annuli $V$, and $T(H) = 0$ for all horizontal annuli $H$.

\end{enumerate}
Then $T$ is a $2$-coboundary.\label{lem2}\end{lemma}

\begin{proof} First, we show that $T$ is a cocycle. Let $E$ be any index $3$ domain---we must show that $\langle T, \partial E \rangle = 0$. For every decomposition $E = D * A$, where $A$ is a vertical or horizontal annulus, there is a corresponding decomposition $E = A * D$, so that $A$ occurs an even number of times in $\partial E$. It now suffices to show that $\partial E$ contains an even number of every other type of index $2$ domain.

To every index $3$ domain $E$ from a generator $x$ to a generator $y$, consider a graph with vertices $x$ at level 3, $y$ at level $0$, and edges down $1$ level corresponding to each way to break off an index $1$ domain (see Figure $\hyperref[fig3]{6}$ for an example of such a graph). Then each level $2$ vertex has an index $2$ domain to $y$, which decomposes into rectangles exactly two ways, so each level $2$ vertex has downward degree $2$, and each level $1$ vertex has an index $2$ domain from $x$, which decomposes into rectangles exactly two ways, so each level $1$ vertex has upward degree $2$. Therefore there are the same number of level $2$ and level $1$ vertices, so since each index $2$ domain that shows up in $\partial E$ corresponds to a level $2$ or $1$ vertex, there are an even number of index $2$ domains. Since an even number of these are annuli, we must therefore have an even number of hexagons. This shows that $\langle T, \partial E \rangle = 0$, as desired.

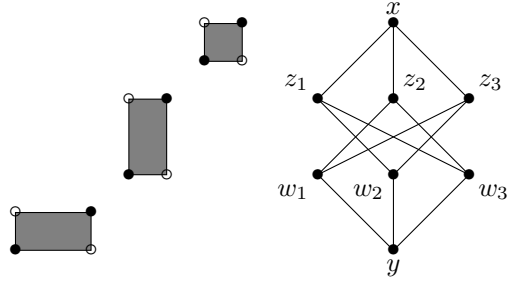
\begin{figure}\begin{tikzpicture} [scale = 0.5]

\draw[fill=gray]    (0, 0) rectangle (2, 1);
\draw[fill=gray]    (3, 2) rectangle (4, 4);
\draw[fill=gray]    (5, 5) rectangle (6, 6);

\node[] at (0, 0) {\textbullet};
\node[] at (2, 1) {\textbullet};
\node[] at (3, 2) {\textbullet};
\node[] at (4, 4) {\textbullet};
\node[] at (5, 5) {\textbullet};
\node[] at (6, 6) {\textbullet};
\node[] at (0, 1) {$\circ$};
\node[] at (2, 0) {$\circ$};
\node[] at (4, 2) {$\circ$};
\node[] at (3, 4) {$\circ$};
\node[] at (6, 5) {$\circ$};
\node[] at (5, 6) {$\circ$};

\node[] at (10, 6) {\textbullet};
\node[] at (8, 4) {\textbullet};
\node[] at (10, 4) {\textbullet};
\node[] at (12, 4) {\textbullet};
\node[] at (8, 2) {\textbullet};
\node[] at (10, 2) {\textbullet};
\node[] at (12, 2) {\textbullet};
\node[] at (10, 0) {\textbullet};

\draw (10, 6) -- (8, 4);
\draw (10, 6) -- (10, 4);
\draw (10, 6) -- (12, 4);
\draw (8, 4) -- (10, 2);
\draw (8, 4) -- (12, 2);
\draw (10, 4) -- (8, 2);
\draw (10, 4) -- (12, 2);
\draw (12, 4) -- (10, 2);
\draw (12, 4) -- (8, 2);
\draw (10, 0) -- (8, 2);
\draw (10, 0) -- (10, 2);
\draw (10, 0) -- (12, 2);

\node[above] at (10, 6) {$x$};
\node[above left] at (8, 4) {$z_1$};
\node[above right] at (10, 4) {$z_2$};
\node[above right] at (12, 4) {$z_3$};
\node[below left] at (8, 2) {$w_1$};
\node[below left] at (10, 2) {$w_2$};
\node[below right] at (12, 2) {$w_3$};
\node[below] at (10, 0) {$y$};

\end{tikzpicture}\label{fig3}\caption{An example of a positive index $3$ domain from a generator $x$ (drawn by \textbullet) to a generator $y$ (drawn by $\circ$), along with the graph defined in the proof of Proposition $\ref{prop3}$. The generators $z_i$ are given by $\circ$ on the $i^{th}$ rectangle from the left and \textbullet \: on the other two, while the generators $w_i$ are given by \textbullet \: on the $i^{th}$ rectangle from the left and $\circ$ on the other two.}.\end{figure}

By Propositions $\ref{prop1}$ and $\ref{prop2}$ it now suffices to show that $T(U) = 0$ where $U$ is the generator of $H_2(\CD_*)$. By definition, $U$ consists of $n$ annuli $A_i$, $n$ annuli $B_i$, $n-2$ hexagons $C_i$, $n - 2$ hexagons $D_i$, $n - 2$ hexagons $E_i$, $\binom{n - 2}{2}$ hexagons $F_{i, j}$, and $\binom{n - 2}{2}$ hexagons $G_{i, j}$, so for any $T$ satisfying the conditions of Lemma $\ref{lem2}$,
\begin{align*}T(U) \equiv n + 3(n-2) + 2 \binom{n - 2}{2} \equiv 0 \pmod 2\end{align*}
so that $T$ is indeed a coboundary.
\end{proof}

\begin{lemma}Let $T$ be the $2$-coboundary from Lemma $\ref{lem2}$. Then $T = \delta s$ if and only if $s$ is a sign assignment.\label{lem3}\end{lemma}

\begin{proof}This is clear from the definitions.\end{proof}

\begin{proof}[Proof of Theorem \ref{prop11}] Existence immediately follows from Lemmas $\ref{lem2}$ and $\ref{lem3}$. For uniqueness, suppose $T = \delta s = \delta s'$. Then $\delta (s - s') = 0$, so $s - s'$ is a $1$-cocycle, which is cohomologous to zero by Proposition $\ref{prop1}$, so there is a $0$-cochain $g$ such that $s = s' + \delta g$.\end{proof}

Given a sign assignment $s$, we can use it to redefine $\CD_*$ in $\Z$ coefficients as follows
\begin{definition}$\CD_*(\G; \Z)$ is freely generated over $\G$ by the positive domains, with the homological grading being the Maslov index. The differential $\partial: \CD_k \rightarrow \CD_{k - 1}$ of $D \in \mathscr{D}^+(x, y)$ is given by
\begin{align*}\partial(D) = \sum\limits_{R * E = D} (-1)^{s(R)} E + (-1)^k \sum\limits_{E * R = D} (-1)^{s(R)} E\end{align*}
where $R$ is a domain of index $1$ from $x$ to some generator $z$ and $E$ is a positive domain from $z$ to $y$.\label{def3}\end{definition}

We now have analogues of Lemma $\ref{lem1}$ and Proposition $\ref{prop1}$ in $\Z$ coefficients, in the following Lemma and Proposition $\ref{prop3}$, respectively.

\begin{lemma}$(\CD_*, \partial)$ is a chain complex.\label{lem4}\end{lemma}

\begin{proof}The proof is similar to the proof of $\ref{lem1}$, except we must keep track of signs.\end{proof}

\begin{proof}[Proof of Proposition \ref{prop3}]The proof is similar to the proof of Proposition $\ref{prop1}$. Specifically, our proof of Proposition $\ref{prop1}$ over $\Z/2$ adapts the proof of \cite[Proposition 3.4]{MS}. This proof is over $\Z$, and a similar adaptation will prove Proposition $\ref{prop3}$.\end{proof}

\section{The Obstruction Complex with Partitions}

The moduli spaces in the construction of the $1$-flow category require more than just positive domains. Since periodic domains (annuli) can bubble, \cite{MS} introduces a new complex to keep track of the bubbles---since there are $n$ different types of bubbles (corresponding to bubbling of the $j^{th}$ horizontal or vertical annulus) which can be at the same or different heights, these correspond to $n$-tuples of ordered partitions.

It is convenient to use both of the following equivalent definitions of an ordered partition of a positive integer $N$ (and when $N = 0$, a partition of $N$ is the empty set).

\begin{itemize}

\item An ordered partition $\lambda$ is a tuple of nonnegative integers $\lambda = (\lambda_1, \dots, \lambda_m)$ such that $N = \sum \lambda_i$. ($m$ is called the $\textit{length}$ of the partition, and is denoted $l(\lambda)$.)

\item An ordered partition $\lambda$ is a tuple $\epsilon(\lambda) = (\epsilon_1(\lambda), \dots, \epsilon_{N - 1}(\lambda)) \in \{ 0, 1 \}^{N - 1}$, where an $\epsilon_i$ equalling $1$ indicates a split at that point. (For instance, the ordered partitions $(1, 1, 1), (1, 2), (2, 1),$ and $(3)$ of $3$ are written $(1, 1), (1, 0), (0, 1),$ and $(0, 0)$, respectively.)

\end{itemize}

Besides annuli bubbling off (the second type of terms that will be in the differential---the first being terms in the differential of $\CD_*$), there are two other boundary degenerations that occur with existing bubbles. Bubbles of the same type may come to the same height (the third type of term), and bubbles may go to height $\pm \infty$ (the fourth and final type of term). The corresponding changes to the partitions can be described as follows:

\begin{definition} The following changes to an ordered partition will describe the differential terms---see \cite[Definitions 4.1, 4.2, 4.3]{MS} for more details.

\begin{itemize}

\item A unit enlargement (at position $k$) increases $N$ by $1$ and adds a $1$ to the tuple $\lambda$ (at position $k$). The set of unit enlargements of $\lambda$ is denoted $\text{UE}(\lambda)$.

\item An elementary coarsening (at position $k$) replaces both terms $\lambda_k$ and $\lambda_{k + 1}$ with one term $\lambda_k + \lambda_{k + 1}$. The set of elementary coarsenings of $\lambda$ is denoted $\text{EC}(\lambda)$.

\item An initial reduction removes $\lambda_1$ (and decreases $N$ by $\lambda_1$), and a final reduction removes $\lambda_m$ (and decreases $N$ by $\lambda_m$). The set of initial reductions (respectively, final reductions) of $\lambda$ is denoted $\text{IR}(\lambda)$ (respectively, $\text{FR}(\lambda)$), where we consider both sets empty if $N = 0$.

\end{itemize}

\label{def4}\end{definition}

We are now ready to define the complex of domains with partitions, $\CDP_*$.

\begin{definition}The complex of positive domains with partitions $\CDP_* = \CDP_*(\G; \Z/2)$ is freely generated by triples of the form $D, \vec{N}, \vec{\lambda}$, where
\begin{itemize}

\item $D \in \mathscr{D}^+(x, y)$ is a positive domain.

\item $\vec{N} \in \N^{n}$ is an $n$-tuple of nonnegative integers, $\vec{N} = (N_1, \dots, N_n)$.

\item $\vec{\lambda} = (\lambda_1, \dots, \lambda_n)$ is an $n$-tuple of ordered partitions, where $\lambda_j = (\lambda_{j, 1}, \dots, \lambda_{j, m_j})$ is an ordered partition of $N_j$.

\end{itemize}
We denote $\lvert \vec{N} \rvert := \sum_{j = 1}^{n} N_j$, and define the total length of $\vec{\lambda}$ to be $\lvert l(\vec{\lambda}) \rvert := \sum_{j = 1}^{n} l(\lambda_j)$. The grading of $(D, \vec{N}, \vec{\lambda})$ is given by the Maslov index of $D$ plus $\lvert l(\vec{\lambda}) \rvert$. The differential is given by the sum of the following four terms.
\begin{itemize}

\item Type I terms, given by taking out a rectangle from $D$, just like in the differential of $\CD_*$.

\item Type II terms, given by taking out a vertical or horizontal annulus passing through $O_j$ from $D$ and performing a unit enlargement to $\lambda_j$.

\item Type III terms, given by an elementary coarsening of one of the partitions $\lambda_j$.

\item Type IV terms, given by taking the initial or final reduction of one of the partitions $\lambda_j$.

\end{itemize}

Precisely, we can write $\partial = \partial_1 + \partial_2 + \partial_3 + \partial_4$ where

\begin{align*}\partial_1(D, \vec{N}, \vec{\lambda}) &= \sum\limits_{R * E = D} (E, \vec{N}, \vec{\lambda}) + \sum\limits_{E * R = D} (E, \vec{N}, \vec{\lambda}) \\ \partial_2(D, \vec{N}, \vec{\lambda}) &= \sum\limits_{j = 1}^{n} \sum\limits_{D = E * H_j \text{ or } E * V_j} \sum\limits_{\lambda_j' \in \text{UE}(\lambda_j)} (E, \vec{N} + \vec{e}_j, \vec{\lambda}') \\ \partial_3(D, \vec{N}, \vec{\lambda}) &= \sum\limits_{j = 1}^{n} \sum\limits_{\lambda_j' \in \text{EC}(\lambda_j)} (D, \vec{N}, \vec{\lambda}') \\ \partial_4(D, \vec{N}, \vec{\lambda}) &= \sum\limits_{j = 1}^{n} \sum\limits_{\lambda_j' \in \text{IR}(\lambda_j)} (D, \vec{N} - \lambda_{j, 1}\vec{e}_j, \vec{\lambda}') + \sum\limits_{j = 1}^{n} \sum\limits_{\lambda_j' \in \text{FR}(\lambda_j)} (D, \vec{N} - \lambda_{j, m_j}\vec{e}_j, \vec{\lambda}')\end{align*}

As in Definition $\ref{def1}$, $R$ is a rectangle, and the annuli $H_j$, $V_j$ are the ones passing through the $j^{th}$ O marking. We also use $\vec{\lambda}' := (\lambda_1, ..., \lambda_{j - 1}, \lambda_j', \lambda_{j + 1}, ..., \lambda_n)$, and $\vec{e}_j := (0, ..., 0, 1, 0, ..., 0)$ with the $1$ in the $j^{th}$ position.\label{def5}\end{definition}

(See \cite[Section 4.2]{MS} for more details.)

It will help us to classify the lower grading generators---that is, generators of $\CDP_0, \CDP_1, \CDP_2, \CDP_3$.

\begin{enumerate}

\item[(0)] $\CDP_0$ is generated by the constant domains with no partitions $(c_{x}, 0, 0)$ for some generator $x$.

\item[(1)] $\CDP_1$ is generated by rectangles with no partitions $(R, 0, 0)$ as well as triples of the form $(c_x, N\vec{e}_j, (N))$ for a constaint domain $c_x$.

\item[(2)] $\CDP_2$ is generated by Maslov index $2$ domains with no partition $(D, 0, 0)$ (for a classification of the kinds of domains $D$, see above or \cite{OSS}), triples of the form $(R, N\vec{e}_j, (N))$ for a rectangle $R$, or a constant domain with partitions of total length $2$. Specifically, we can have triples of the form $(c_x, N\vec{e}_j + M \vec{e}_k, ((N), (M)))$ (where $j \neq k$), or $(c_x, (N + M)\vec{e}_j, (N, M))$.

\item[(3)] Finally, $\CDP_3$ is generated by Maslov index $3$ domains with no partition, Maslov index $2$ domains with a partition of the form $(D, N\vec{e}_j, (N))$, rectangles with a partition of total length $2$, and constant domains with partitions of total length $3$, which has the following cases:

\begin{itemize}

\item $(c_x, N_j \vec{e}_j + N_k \vec{e}_k + N_l \vec{e}_l, ((N_j), (N_k), (N_l))$ for $j, k, l$ distinct.

\item $(c_x, (N_j + M_j) \vec{e}_j + N_k \vec{e}_k, ((N_j, M_j), (N_k))$ for $j, k$ distinct.

\item $(c_x, (N_j + M_j + P_j) \vec{e}_j, (N_j, M_j, P_j))$.

\end{itemize}\end{enumerate}

\begin{lemma}$(\CDP_*, \partial)$ is a chain complex.\label{lem5}\end{lemma}

\begin{proof} The proof follows a similar case analysis to \cite[Lemma 4.4]{MS}. Write $\partial = \partial_1 + \partial_2 + \partial_3 + \partial_4$, where $\partial_k$ is the type $k$ term in the differential. Since $\partial_1$ is just the differential from $\CD_*$, we have by Lemma $\ref{lem1}$ that $\partial_1^2 = 0$. Now for $\partial_2^2$, the terms will correspond to removing two annuli (and doing two unit enlargements). If the annuli pass through two different $O_i$ and $O_j$, then the corresponding term shows up twice, once in each order. If the annuli pass through the same $O_j$, then the corresponding term also shows up twice---once for each order in doing the unit enlargements. Therefore, $\partial_2^2 = 0$. We can similarly show that
\begin{align*}
&\partial_3^2 = 0 \text{ and} \\& \partial_1 \partial_2 + \partial_2 \partial_1 = 0 \text{ and} \\&\partial_1 \partial_3 + \partial_3 \partial_1 = 0 \text{ and} \\& \partial_2 \partial_3 + \partial_3 \partial_2 = 0 \text{ and} \\& \partial_1 \partial_4 + \partial_4 \partial_1 = 0
\end{align*}
by doing the respective operations in two different orders.

Now consider $\partial_2 \partial_4 + \partial_4 \partial_2$, the terms of which correspond to a unit enlargement and an initial or final reduction, in either order. If one is done to $\lambda_i$ and another to $\lambda_j$ where $i \neq j$, then the two commute and cancel just like before. If both are done to $\lambda_i$, then all terms follow one of these cases:

\begin{itemize}

\item A unit enlargement not at the beginning, followed by an initial reduction. This cancels with the initial reduction followed by doing the enlargement one place earlier.

\item A unit enlargement not at the end, followed by a final reduction. This cancels with the final reduction followed by the same enlargement.

\item A unit enlargement at the beginning, followed by an initial reduction; or a unit enlargement at the end, followed by a final reduction. These cancel with each other.

\end{itemize}

Finally, consider the last terms of $\partial^2$, $\partial_4^2 + \partial_3 \partial_4 + \partial_4 \partial_3$. Again there are some special types of terms:
\begin{itemize}

\item The elementary coarsening of $\lambda_i$ by combining the first two parts, followed by a initial reduction of $\lambda_i$, cancels with two initial reductions of $\lambda_i$.

\item The elementary coarsening of $\lambda_i$ by combining the last two parts, followed by a final reduction of $\lambda_i$, cancels with two final reductions of $\lambda_i$.

\end{itemize}
where all the other terms cancel by doing the operations in the two different orders. \end{proof}

We would like to compute the homology of $\CDP_*$ using successive filtrations, as in the proof of Proposition $\ref{prop1}$.


\begin{prop}There is a filtration on $\CDP_*$ such that the associated graded has homology $(\Z/2)^{2^n} \otimes (\Z/2)[U]$.\label{prop4}\end{prop}

\begin{proof}We again follow the proof of \cite[Proposition 4.6]{MS}. We can filter the complex $\CDP_*'$ in several steps. First we filter $\CDP_*$ by the quantity $A(D) \in \N^n$ which are the coefficients of $D$ in the rightmost column. As in the proof of Proposition $\ref{prop1}$, we can assume without loss of generality that the minimum of $A(D)$ occurs in the top right corner, and then filter the associated graded $\CDP_*^a$ by $B(D) \in \N^{n - 1}$ which are the coefficients of $D$ in the topmost row. In the associated graded $\CDP_*^{a, b}$, there are no type II terms in the differential, since such terms must decrease either $A$ or $B$. Since $\lvert \vec{N} \rvert$ is kept constant by type I and III terms and decreased by type IV terms, it is a filtration on $\CDP_*^{a, b}$, so filtering by $\lvert \vec{N} \rvert$ and (as in the proof of Proposition $\ref{prop1}$) the end generator $y$ gives a direct sum of complexes $\CDP_*^{a, b, y, \vec{N}}$.

When $a \neq (l, l, \dots, l)$ or $b \neq 0$ or $y \neq x^{\Id}$, filtering by the total length of $\vec{\lambda}$ removes all type III terms and keeps all type I terms, so $\CDP_*^{a, b, y, \vec{N}}$ is a direct sum of complexes $\CD_*^{a, b, y}$ which were all shown to be acyclic in the proof of \cite[Proposition 3.4]{MS}. Additionally, when $a = (l, l, \dots, l)$, $b = 0$, $y = x^{\Id}$, and at least one $N_j > 1$, every generator of $\CDP_*^{a, b, y, \vec{N}}$ is represented by some $(D, \vec{N}, \vec{\lambda})$ where $D = kV_n$, so we only have type III terms. The partitions of $N_j$ are given by $(\epsilon_1, \dots, \epsilon_{N_{j} - 1})$, where the elementary coarsenings just change a $1$ to a $0$. This gives a hypercube complex, which is acyclic. Therefore, we are only left with the associated graded complexes $\CDP^{a, b, y, \vec{N}}$ where $a = (l, l, \dots, l)$, $b = 0$, $y = x^{Id}$, and every $N_j$ is $0$ or $1$.\end{proof} 

\begin{corollary}$H_k(\CDP_*; \Z/2)$ has rank at most
\begin{align*}\sum\limits_{l = 0}^{\lfloor k/2 \rfloor} \binom{n}{k - 2l}\end{align*}\end{corollary}




In the proof of Proposition $\ref{prop2}$, we found a cocycle that detects the generator of $H_2(\CD_*)$. We will use a similar procedure to compute $H_0(\CDP_*)$ through $H_3(\CDP_*)$.

\begin{proof}[Proof of Theorem $\ref{prop8}$] \textbf{(k = 0)} This case is clear.

\textbf{(k = 1)} The $n$ generators of $H_1(\text{AssGr}(\CDP_*))$ are the triples 
\begin{align*}g_j := (c_{x^{\Id}}, \vec{e}_j, (1))\end{align*}
which are still cycles in $\CDP_1$ (because their initial and final reductions cancel). It will suffice to show that there exist $n$ $1$-cocycles $r_j$ such that $r_j(g_k) = 1$ if and only if $j = k$. Let $f_j$ be the $1$-cochain in $\CD_*$ such that $\delta f_j(D) = 1$ if and only if $D$ is the vertical annulus $V_j$ or the horizontal annulus $H_j$---$f_j$ exists by Proposition $\ref{prop2}$ because the $2$-cocycle which is $1$ on $V_j$ and $H_j$ and zero on every other index $2$ domain is a coboundary, since it is zero on the generator $U$ of $H_2(\CD_*)$. We can extend $f_j$ to $\CDP_*$ by setting it equal to zero on all triples $(c_x, N\vec{e}_j, (N))$. Let $N_j$ be the $1$-cocycle that is the value of $N_j$ in the triple $(D, \vec{N}, \vec{\lambda})$, and 
\begin{align*}r_j := N_j + f_j.\end{align*}
To show that $r_j$ is a cocycle, we consider all possible triples $(D, \vec{N}, \vec{\lambda})$ in grading $2$. If $N_j = 0$ and $D$ is not the annulus $V_j$ or $H_j$, then by definition $\delta r_j(D, \vec{N}, \vec{\lambda}) = 0$. If $N_j = 0$ and $D = V_j$ or $H_j$, then $\vec{N} = 0$ and $\delta r_j(D, 0, 0) = N_j(c_{x}, \vec{e}_j, (1)) + (f_j(R_1) + f_j(R_2)) = 1 + 1 = 0 \pmod{2}$, where $D = R_1 * R_2$ is the decomposition into rectangles. Finally, if $N_j = M > 0$, there are three cases:
\begin{itemize}

\item $D$ is a rectangle. In this case $\vec{N} = M\vec{e}_j$ and $\lambda_i = (M)$, so the initial and final reduction of $\lambda_i$ cancel, and the only other differential terms are removing $D$. If $D$ is a rectangle from $x$ to $y$, then $\delta r_i(D, M\vec{e}_j, (M)) = N_j(c_{x}, M\vec{e}_j, (M)) + N_j(c_{y}, M\vec{e}_j, (M)) = M + M = 0 \pmod{2}$.

\item Some $N_k > 0$, where $k \neq j$. Then $D$ must be a constant domain, and both $\lambda_j$ and $\lambda_k$ are length $1$ partitions, so their initial and final reductions all cancel.

\item $D$ is a constant domain and $\lambda_j = (M_1, M_2)$ is a length $2$ partition. In this case we have all of the type III and IV differentials, which gives
\begin{align*}&\delta r_j(c_x, (M + N)\vec{e}_j, (M, N)) \\&= N_j(c_x, M\vec{e}_j, (M)) + N_j(c_x, N\vec{e}_j, (N)) + N_j(c_x, (M + N)\vec{e}_j, (M + N)) \\&= M + N + (M + N) = 0 \pmod{2}\end{align*}

\end{itemize}






Therefore $r_j$ is a cocycle for each $j$, and by definition $r_j(g_j) = 1$ if and only if $j = k$, so the $g_j$ are in fact the generators of $H_1(\CDP_*)$.

\textbf{(k = 2)} We similarly consider $2$-cocycles. $(\Z/2)^{\binom{n}{2}}$ of the generators of $H_2(\text{AssGr}(\CDP_*))$ are the triples 
\begin{align*}g_{j, k} := (c_{x^{\Id}}, \vec{e}_j + \vec{e}_k, ((1), (1)))\end{align*}
which are similarly still cycles in $\CDP_2$. The final generator will be given by a slight modification $U'$ of 
$(U, 0, 0)$, where $U$ is the generator of $H_2(\CD_*)$. The boundary of $U$ in $\CDP_*$ contains only pairs of triples of the form
\begin{align*}(c_{x^j}, \vec{e}_j, (1)) \text{ and } (c_{y^j}, \vec{e}_j, (1))\end{align*}
corresponding to type II differentials on the annuli $A_j$ and $B_j$. Here $x^j$ and $y^j$ are the generators $[n23\dots(n-j)1(n-i+1)\dots(n-1)]$ and $[(n-j+1)23\dots(n-j)(n-j+2)\dots n1]$, respectively, as depicted in Figure $\ref{fig2}$. For each $j$, the generators $x^j$ and $y^j$ each have a planar domain (that is, a domain that does not intersect the topmost row or rightmost column of the grid), $D_{j, 1}$ and $D_{j, 2}$ respectively, from the identity generator $x^{\Id}$. From Figure $\ref{fig2}$, we see that $D_{j, 2}$ is the reflection of $D_{j, 1}$ about the diagonal from the bottom left to the top right of the grid, so that $(-D_{j, 1}) * D_{j, 2}$ is an even index planar domain from $x^j$ to $y^j$. This domain decomposes into an even number of planar rectangles $\pm R_{jk}$ (where each $R_{jk}$ is positive), so that adding each $(R_{jk}, \vec{e}_j, (1))$ to $(U, 0, 0)$ will cancel the rest of its boundary, making a cycle $U'$.

In the proof of Proposition $\ref{prop2}$, we used the $2$-cocycle $r$ which is $1$ on the rightmost vertical annulus and zero on every other $2$-chain. Extending $r$ to $\CDP_*$ by setting it equal to zero on every $2$-chain with $\lvert \vec{N} \rvert > 0$ still gives a cocycle, since $\CDP_*$ has no new ways to create an annulus in the boundary, and we still have that $r(U') = 1$, while all of the $r(g_{j, k}) = 0$. Now it suffices to find $r_{j, k}$ such that $r_{j, k}(U') = 0$ for all $j, k$, and $r_{j, k}(g_{l, m}) = 1$ if and only if $\{ l, m \} = \{ j, k \}$. Let $f_j$ be the $1$-cocycles defined in the proof of (1), and let
\begin{align*}f_{j}^k(R, \vec{N}, \vec{\lambda}) = N_k f_j(R)\end{align*}
where $R$ is a rectangle and $\vec{\lambda}$ has total length $1$ (and $f_j^k = 0$ on all other $2$-chains). Now let $N_j N_k$ be the $2$-cocycle that is the product of the values of $N_j$ and $N_k$ for a triple $(D, \vec{N}, \vec{\lambda})$, and let 
\begin{align*}r_{j, k} := N_j N_k + f_j^k + f_k^j.\end{align*}
To show that $r_{j, k}$ is a cocycle, we consider all possible triples $(D, \vec{N}, \vec{\lambda})$ in grading $3$. If $N_j = 0$ (respectively, $N_k = 0$) and $D$ does not contain the annulus $V_j$ or $H_j$ (respectively, $V_k$ or $H_k$), then by definition $\delta r_{j, k}(D, \vec{N}, \vec{\lambda}) = 0$. If $N_j = 0$, $N_k > 0$ (or vice versa), and $D$ contains $V_j$ or $H_j$, then $D = V_j$ or $H_j$ and all $N_l = 0$ for $l \neq k$, so that $\delta r_{j, k}(D, \vec{N}, \vec{\lambda}) = 0$ similarly to the proof of (1). Finally, if $N_j = M_j > 0$ and $N_k = M_k > 0$, there are three cases:
\begin{itemize}

\item $D$ is a rectangle. In this case $\vec{N} = M_j\vec{e}_j + M_k\vec{e}_k, \lambda_j = (M_j),$ and $\lambda_k = (M_k)$, so the initial and final reductions of $\lambda_j$ and $\lambda_k$ cancel, and the only other differential terms are removing $D$. If $D$ is a rectangle from $x$ to $y$, then $\delta r_{j, k}(D, \vec{N}, \vec{\lambda}) = N_jN_k(c_{x}, M_j\vec{e}_j + M_k\vec{e}_k, ((M_j), (M_k))) + N_iN_j(c_{y}, M_j\vec{e}_j + M_k\vec{e}_k, ((M_j), (M_k))) = M_i M_j + M_i M_j = 0 \pmod{2}$.

\item Some $N_l > 0$, where $l \neq j, k$. Then $D$ must be a constant domain, and all of $\lambda_j$, $\lambda_k$, and $\lambda_l$ are length $1$ partitions, so their initial and final reductions all cancel.

\item $D$ is a constant domain and $\lambda_j = (M_{j,1}, M_{j,2})$ is a length $2$ partition (or symmetrically, $\lambda_k = (M_{k,1}, M_{k,2})$). In this case the initial and final reductions of $\lambda_k$ cancel, but we have all of the type III and IV differentials of $\lambda_j$, which give
\begin{align*}&\delta r_{j, k}(c_x, (M_{j, 1} + M_{j, 2})\vec{e}_j + M_k\vec{e}_k, ((M_{j, 1}, M_{j, 2}), (M_k))) \\&= N_jN_k(c_x, M_{j, 1}\vec{e}_j + M_k\vec{e}_k, ((M_{j, 1}), (M_k))) \\&+ N_jN_k(c_x, M_{j, 2}\vec{e}_j + M_k\vec{e}_k, (M_{j, 2}), (M_k))) \\&+ N_jN_k(c_x, (M_{j, 1} + M_{j, 2})\vec{e}_j + M_k\vec{e}_k, ((M_{j, 1} + M_{j, 2}), (M_k))) \\&= M_k(M_{j, 1} + M_{j, 2} + (M_{j, 1} + M_{j, 2})) = 0 \pmod{2}\end{align*}

\end{itemize}

Therefore $r_{j, k}$ is a cocycle for all $j, k$, and by definition $r_{j, k}(U) = 0$. $U'$ only adds an even number of planar rectangles with partitions that can contribute to $f_j^k$ or $f_k^j$, no two of which can ever form an annulus, so $r_{j, k}(U') = 0$. Finally, by definition $r_{j, k}(g_{l, m}) = 1$ if and only if $\{ l, m \} = \{ j, k \}$, so these (along with $U'$) are in fact the generators of $H_2(\CDP_*)$.

\textbf{(k = 3)} $\binom{n}{3}$ of the generators of $H_3(\text{AssGr}(\CDP_*))$ are the triples 
\begin{align*}g_{j, k, l} := (c_{x^{\Id}}, \vec{e}_j + \vec{e}_k + \vec{e}_l, ((1), (1), (1)))\end{align*}
which are similarly still cycles in $\CDP_3$. The other $n$ generators are the triples $U_j'$ obtained from $U'$ by performing unit enlargements on $N_j$ and adding the triples $(R_{jk}, 2\vec{e}_j, (2))$ defined previously (for this fixed $j$). Let $V_{(n)}$ be the rightmost vertical annulus and define the cochain
\begin{align*}rr_j(D, \vec{N}, \vec{\lambda}) := \begin{cases}0 & D\text{ does not contain } V_{(n)} \\ r_j(D * -V_{(n)}, \vec{N}, \vec{\lambda}) & D\text{ contains } V_{(n)}\end{cases}\end{align*}
where $r_j$ is the $1$-cocycle from the proof of (1). To show that $rr_j$ are cocycles, we consider all $\delta rr_j(D, \vec{N}, \vec{\lambda})$ for triples $(D, \vec{N}, \vec{\lambda}) \in \CDP_4$. If $D$ does not contain $V_{(n)}$, then this quantity is zero by definition. If $D$ contains $V_{(n)}$, then its Maslov index is at least 2, so that we have the following cases:
\begin{itemize}

\item $D$ is an index $4$ domain. In this case, $\vec{N} = 0$, so let $E = D * (-V_{(n)})$. If $E$ is also an annulus $V_k$, then
\begin{align*}\delta rr_j(D, \vec{N}, \vec{\lambda}) &= rr_j(A_1 * V_{(n)}, 0, 0) + rr_j(A_2 * V_{(n)}, 0, 0) + rr_j(E * B_1, 0, 0) + rr_j(E * B_2, 0, 0) \\&+ rr_j(V_{(n)}, \vec{e}_k, (1)) + rr_j(E, \vec{e}_n, (1)) \text{ where } A_1 * A_2 = E, B_1 * B_2 = V_{(n)} \\&= r_j(A_1, 0, 0) + r_j(A_2, 0, 0) + r_j(c_x, \vec{e}_k, (1)) \\&+ (r_j(B_1, 0, 0) + r_j(B_2, 0, 0) + r_j(c_x, \vec{e}_n, (1))) \: (\text{added iff } k = n) = 0\end{align*}
by definition of $r_j$ (since this is just $\delta r_j(V_k)$ with possibly $\delta r_j(V_{(n)})$ added if $k = n$). If $E$ is not an annulus, we similarly have
\begin{align*}\delta rr_j(D, \vec{N}, \vec{\lambda}) &= rr_j(A_1 * V_{(n)}, 0, 0) + rr_j(A_2 * V_{(n)}, 0, 0) + rr_j(E * B_1, 0, 0) + rr_j(E * B_2, 0, 0) \\&+ rr_j(E, \vec{e}_n, (1)) \text{ where } A_1 * A_2 = E, B_1 * B_2 = V_{(n)} \\&= r_j(A_1, 0, 0) + r_j(A_2, 0, 0) + r_j(c_x, \vec{e}_k, (1)) = 0\end{align*}

\item $D = R * V_{(n)}$ is an index $3$ domain, where $R$ is a rectangle from a generator $x$ to a generator $y$. In this case, $\vec{N} = N\vec{e}_k$ and $\vec{\lambda} = \lambda_k = (N)$. Suppose that in the domain $D$, $V_{(n)} = A * B$, and $R = B$ (or symmetrically, $R = A$). In this case,
\begin{align*}\delta rr_j(D, \vec{N}, \vec{\lambda}) &= rr_j(R * A, N \vec{e}_k, (N)) + rr_j(A * R, N \vec{e}_k, (N)) + rr_j(R, N \vec{e}_k + e_n, ((N), (1))) \\&= r_j(c_x, N \vec{e}_k, (N)) + r_j(c_y, N \vec{e}_k, (N)) = 0 \text{ by definition of } r_j\end{align*}
and if $R$ is not $A$ or $B$, we similarly have
\begin{align*}\delta rr_j(D, \vec{N}, \vec{\lambda}) &= rr_j(R * A, N \vec{e}_k, (N)) + rr_j(R * B, N \vec{e}_k, (N)) + rr_j(V_{(n)}, N \vec{e}_k, (N)) \: (\text{at the generator } x) \\&+ rr_j(V_{(n)}, N \vec{e}_k, (N)) \: (\text{at } y) + rr_j(R, N \vec{e}_k + \vec{e}_n, ((N), (1))) \\&= r_j(c_x, N \vec{e}_k, (N)) + r_j(c_y, N \vec{e}_k, (N)) = 0\end{align*}

\item $D$ is an index $2$ domain. In this case, $D = V_{(n)}$, $\lvert \vec{\lambda} \rvert = 2$, and $\delta rr_j(D, \vec{N}, \vec{\lambda}) = \delta r_j(c_x, \vec{N}, \vec{\lambda})$ since the type I and II terms cannot possibly have an annulus---note that we previously showed this expression to be zero in showing that $r_j$ is a cocycle.

\end{itemize}
Therefore $rr_j$ is a cocycle, and by construction $rr_j(U_k') = 1$ if and only if $k = j$, and all of the $rr_j(g_{k, l, m}) = 0$. It remains to find cocycles $r_{j, k, l}$ such that $r_{j, k, l}(U_m') = 0$ for all $j, k, l, m$, and $r_{j, k, l}(g_{m, p, q}) = 1$ if and only if $\{ m, p, q \} = \{ j, k, l \}$. Let $f_j$ be the $1$-cocycles defined in the proof of (1), and let
\begin{align*}f_{j}^{k, l}(R, \vec{N}, \vec{\lambda}) = N_k N_l f_j(R)\end{align*}
where $R$ is a rectangle and $\lambda_k, \lambda_l$ are both length $1$ partitions (and $f_j^{k, l} = 0$ on all other $3$-chains). Now let $N_j N_k N_l$ be the $3$-cocycle that is the product of the values of $N_j, N_k,$ and $N_l$ for a triple $(D, \vec{N}, \vec{\lambda})$, and let 
\begin{align*}r_{j, k, l} = N_j N_k N_l + f_j^{k, l} + f_k^{j, l} + f_l^{j, k}.\end{align*}
To show that $r_{j, k, l}$ is a cocycle, we consider all possible triples $(D, \vec{N}, \vec{\lambda})$ in grading $4$. If $N_j = 0$ (respectively, $N_k = 0$ or $N_l = 0$) and $D$ does not contain the annulus $V_j$ or $H_j$ (respectively, $V_k$ or $H_k$, or $V_l$ or $H_l$), then by definition $\delta r_{j, k, l}(D, \vec{N}, \vec{\lambda}) = 0$. If $N_j = 0$, $N_k > 0$, and $N_l > 0$ (or symmetrically, any other case where exactly one is zero), and $D$ contains $V_j$ or $H_j$, then $D = V_j$ or $H_j$ and all $N_m = 0$ for $m \neq k, l$, so that $\delta r_{j, k, l}(D, \vec{N}, \vec{\lambda}) = 0$ similarly to the proof of (1). Finally, if $N_j = M_j > 0, N_k = M_k > 0,$ and $N_l = M_l > 0$, there are three cases:
\begin{itemize}

\item $D$ is a rectangle. In this case $\vec{N} = M_j\vec{e}_j + M_k\vec{e}_k + M_l\vec{e}_l, \lambda_j = (M_j), \lambda_k = (M_k)$ and $\lambda_l = (M_l)$, so the initial and final reductions of $\lambda_j, \lambda_k$ and $\lambda_l$ cancel, and the only other differential terms are removing $D$. If $D$ is a rectangle from $x$ to $y$, then 
\begin{align*}\delta r_{j, k, l}(D, \vec{N}, \vec{\lambda}) &= N_jN_kN_l(c_{x}, M_j\vec{e}_j + M_k\vec{e}_k + M_l\vec{e}_l, ((M_j), (M_k), (M_l))) \\&+ N_jN_kN_l(c_{y}, M_j\vec{e}_j + M_k\vec{e}_k + M_l\vec{e}_l, ((M_j), (M_k), (M_l))) \\&= M_j M_k M_l + M_j M_k M_l = 0 \pmod{2}.\end{align*}

\item Some $N_m > 0$, where $m \neq j, k, l$. Then $D$ must be a constant domain, and all of $\lambda_j$, $\lambda_k$, $\lambda_l$, and $\lambda_m$ are length $1$ partitions, so their initial and final reductions all cancel.

\item $D$ is a constant domain and $\lambda_j = (M_{j,1}, M_{j,2})$ is a length $2$ partition (or symmetrically, $\lambda_k = (M_{k,1}, M_{k,2})$ or $\lambda_l = (M_{l,1}, M_{l,2})$). In this case the initial and final reductions of $\lambda_k$ and $\lambda_l$ cancel, but we have all of the type III and IV differentials of $\lambda_j$, which give
\begin{align*}&\delta r_{j, k, l}(c_x, (M_{j, 1} + M_{j, 2})\vec{e}_j + M_k\vec{e}_k + M_l\vec{e}_l, ((M_{j, 1}, M_{j, 2}), (M_k), (M_l))) \\&= N_jN_kN_l(c_x, M_{j, 1}\vec{e}_j + M_k\vec{e}_k + M_l\vec{e}_l, ((M_{j, 1}), (M_k), (M_l))) \\&+ N_jN_kN_l(c_x, M_{j, 2}\vec{e}_j + M_k\vec{e}_k + M_l\vec{e}_l, (M_{j, 2}), (M_k), (M_l))) \\&+ N_jN_kN_l(c_x, (M_{j, 1} + M_{j, 2})\vec{e}_j + M_k\vec{e}_k + M_l\vec{e}_l, ((M_{j, 1} + M_{j, 2}), (M_k), (M_l))) \\&= M_kM_l(M_{j, 1} + M_{j, 2} + (M_{j, 1} + M_{j, 2})) = 0 \pmod{2}\end{align*}

\end{itemize}

Therefore $r_{j, k, l}$ are cocycles, satisfying $r_{j, k, l}(F_{m}'') = 0$ for all $j, k, l, m$, and $\\r_{j, k, l}(c_{x^{\Id}}, \vec{e}_m + \vec{e}_p + \vec{e}_q, ((1), (1))) = 1$ if and only if $\{ m, p, q \} = \{ j, k, l \}$.\end{proof}

\section{Sign Assignments for Domains with Partitions}

Similarly to Section $3$, we find the criteria that the coboundary of a sign assignment for $\CDP_*$ must satisfy. 

\begin{definition}A sign assigment for $\CDP_*$ is a $1$-cochain $s$ on $\CDP_*$ such that
\begin{enumerate}

\item $\delta s(D, 0, 0) = 1$ for any index $2$ domain $D$ that is not an annulus.


\item $\delta s(V, 0, 0) = 1$ for any vertical annulus $V$, and $\delta s(H, 0, 0) = 0$ for any horizontal annulus $H$.

\item $\delta s(R, (0, N\vec{e}_j, (N)) = 0$ for any rectangle $R$, any $N > 0$, and any $j$.

\item $\delta s(c_{x}, N\vec{e}_j + M\vec{e}_k, ((N), (M))) = 0$ for any constant domain $c_x$, any $N, M > 0$, and any $j, k$.

\item $\delta s(c_{x}, (N + M)\vec{e}_j, ((N, M))) = 0$ for any constant domain $c_x$, any $N, M > 0$, and any $j$.

\end{enumerate} \label{def7}\end{definition}

\begin{proof}[Proof of Theorem \ref{prop6}] Let $T$ be the $2$-cochain with values given by (1)--(5) of Definition $\ref{def7}$. To see that $T$ is a cocycle, we evaluate $\delta T$ on all triples $(D, \vec{N}, \vec{\lambda})$ in grading $3$. These are given by
\begin{itemize}

\item $(D, 0, 0)$ where $D$ is an index $3$ domain. The proof of Lemma $\ref{lem3}$ shows that the contributions to $\delta T$ by Type I differential terms all cancel. If $D$ does not contain an annulus, these are all the differential terms. If $D$ does contain an annulus $A = H_j$ or $V_j$, we can write $D = R * A$ for a rectangle $R$, so that the type II differential term gives $(R, (0, \vec{e}_j, (1))$, which does not contribute to $\delta T$ by (3).

\item $(D, N\vec{e}_j, (N))$ where $D$ is an index $2$ domain. Here the initial and final reduction of the partition both give $(D, 0, 0)$ so their contributions to $\delta T$ cancel. The decompositions of $D$ into rectangles do not contribute to $\delta T$ by condition (3), and again if $D$ is not an annulus, then these are the only other boundary terms. If $D$ is an annulus, then either $D = H_j, D = V_j$, or $D$ is some other annulus $H_k$ or $V_k$. In the latter case, the type II differential term gives $(c_x, N\vec{e}_j + \vec{e}_k, ((N), (1)))$ which does not contribute to $\delta T$ by (4). In the former case, the type II differential gives two terms, $(c_x, (N + 1)\vec{e}_j), (1, N))$ and $(c_x, (N + 1)\vec{e}_j, (N, 1))$, which do not contribute to $\delta T$ by (5).

\item $(R, \vec{N}, \vec{\lambda})$ where $R \in \mathscr{D}^+(x, y)$ is a rectangle and $\vec{\lambda}$ has total length $2$. Here the type $I$ differential removes $R$ two ways, which leaves either $(c_x, M\vec{e}_j + N\vec{e}_k, ((M), (N)))$ (and the corresponding term for $c_y$, which do not contribute to $\delta T$ by (4)), or $(c_x, (M + N)\vec{e}_j), (M, N))$ (and the corresponding term for $c_y$, which do not contribute by (5)). All type $III$ and $IV$ terms do not contribute by condition (3). Since $R$ cannot possibly contain an annulus, there are no further terms so $\delta T(R, \vec{N}, \vec{\lambda}) = 0$.

\item $(c_x, \vec{N}, \vec{\lambda})$ where $c_x$ is a constant domain and $\vec{\lambda}$ has total length $3$. None of these terms contribute to $\delta T$ by (4) and (5).






\end{itemize}
Hence $T$ is a cocycle, so it remains to show $T$ is zero on every generator of $H_2(\CDP_*)$ listed in the proof of Theorem $\ref{prop8}$. By definition, every $T(c_{x}, \vec{e}_j + \vec{e}_k, ((1), (1))) = 0$. Also, $T(U, 0, 0) = 0$ by Lemma $\ref{lem2}$, so $T(U') = 0$ by condition (3), since these are the only types of triples added to $(U, 0, 0)$. Therefore $T$ must be the zero cocycle by Theorem $\ref{prop8}$, so $T = \delta s$ for some $s$. The values $s_j$ uniquely determine the $H^1(\CDP_*)$ class of $s$ by Theorem $\ref{prop8}$, so at that point $s$ is unique up to gauge equivalence (like sign assignments for $\CD_*$).\end{proof}

There are two types of triples in grading $1$---rectangles with no partitions and constant domains with a length $1$ partition. By uniqueness, the sign of a rectangle with no partition in $\CDP_*$ agrees with the sign of that rectangle in $\CD_*$, so it remains to compute the signs of constant domains with a length $1$ partition.

\begin{prop}For any constant domain $c_x$ and any $N > 0$,
\begin{align*}s(c_x, N\vec{e}_j, (N)) = Ns_j \pmod 2\end{align*}\label{prop7}\end{prop}

\begin{proof}We first show that the sign is independent of the generator $x$. Let $R \in \mathscr{D}^+(x, y)$ be a rectangle. By (3) of Definition $\ref{def7}$,
\begin{align*}0 = \delta s(R, N\vec{e}_j, (N)) = s(c_x, N\vec{e}_j, (N)) + s(R, 0, 0) + s(R, 0, 0) + s(c_y, N\vec{e}_j, (N))\end{align*}
so that $s(c_x, N\vec{e}_j, (N)) = s(c_y, N\vec{e}_j, (N))$, and given any domain from $x$ to $y$, we find a decomposition into rectangles and repeatedly apply this equation. Therefore we can now assume without loss of generality that $x = x^{\Id}$. We will use the uniqueness of $s$ up to the values $s_j$ to proceed by induction on $N$. The base case is clear, and by (5) of Definition $\ref{def7}$ we must have that
\begin{align*}0 = \delta s(c_x, N\vec{e}_j, (1, N-1)) &= s(c_x, \vec{e}_j, (1)) + s(c_x, (N - 1)\vec{e}_j), (N - 1)) + s(c_x, N\vec{e}_j, (N)) \\&= s_j + (N - 1)s_j \pmod 2\end{align*}
by the inductive hypothesis, so that $s(c_x, N\vec{e}_j, (N)) = Ns_j \pmod 2$, which completes the induction.\end{proof}

\begin{remark}It would suffice by uniqueness to define a sign assignment on $\CDP_*$ by defining a sign assignment on $\CD_*$ and extending it by Proposition $\ref{prop7}$. Doing so would give another proof of Proposition $\ref{prop6}$.\end{remark}

Again, now that we have a sign assignment $s$, we can extend $\CDP_*$ to $\Z$ coefficients. 
As in $\CD_*$, the sign associated to breaking off a rectangle is the sign of the rectangle $s(R)$ given by the sign assignment. We now describe the sign of the other differential terms.

\begin{definition}Let $s$ be a sign assignment for $\CDP_*$.

\begin{itemize}

\item Given an ordered partition $\lambda$ and the unit enlargement $\lambda' = (\lambda_1, ..., \lambda_{k - 1}, 1, \lambda_k, ..., \lambda_m)$, the sign of the unit enlargement is 
\begin{align*}s(\lambda, \lambda') = k + 1 \pmod 2.\end{align*}

\item Given an ordered partition $\lambda$ and the elementary coarsening \\$\lambda' = (\lambda_1, ..., \lambda_{k - 1}, \lambda_k + \lambda_{k + 1}, \lambda_{k + 2}, ..., \lambda_m)$, the sign of the elementary coarsening is
\begin{align*}s(\lambda, \lambda') = k \pmod 2.\end{align*}

\item Given an ordered partition $\lambda = (\lambda_1, ..., \lambda_m)$ and its initial reduction $\lambda'$, the sign of the reduction is given by
\begin{align*}s(\lambda, \lambda') = \lambda_1 s_j \pmod 2\end{align*}
and the sign of its final reduction is given by the same expression, with $\lambda_m$ replacing $\lambda_1$.

\end{itemize}\label{def9}\end{definition}

\begin{definition}The complex of positive domains with partitions $\CDP_* = \CDP_*(\G; \Z)$ is freely generated by triples of the form $D, \vec{N}, \vec{\lambda}$, where
\begin{itemize}

\item $D \in \mathscr{D}^+(x, y)$ is a positive domain.

\item $\vec{N} \in \N^{n}$ is an $n$-tuple of nonnegative integers, $\vec{N} = (N_1, \dots, N_n)$.

\item $\vec{\lambda} = (\lambda_1, \dots, \lambda_n)$ is an $n$-tuple of ordered partitions, where $\lambda_j = (\lambda_{j, 1}, \dots, \lambda_{j, m_j})$ is an ordered partition of $N_j$.

\end{itemize}
The grading of $(D, \vec{N}, \vec{\lambda})$ is given by the Maslov index of $D$ plus the sum of the lengths of the $\lambda_j$ (which is referred to as the total length of $\vec{\lambda}$). The differential is given by four terms, $\partial = \partial_1 + \partial_2 + \partial_3 + \partial_4$, where

\begin{align*}\partial_1(D, \vec{N}, \vec{\lambda}) &= \sum\limits_{R * E = D} (-1)^{s(R)}(E, \vec{N}, \vec{\lambda}) + (-1)^{\mu(D)}\sum\limits_{E * R = D} (-1)^{s(R)}(E, \vec{N}, \vec{\lambda}) \\ \partial_2(D, \vec{N}, \vec{\lambda}) &= (-1)^{\mu(D)}\sum\limits_{j = 1}^{n} (-1)^{l(\lambda_1) + \dots + l(\lambda_{j - 1})} \sum\limits_{D = E * H_j \text{ (horizontal)}} (-1)^{1 + s(\lambda_j, \lambda_j')} \sum\limits_{\lambda_j' \in \text{UE}(\lambda_j)} (E, \vec{N} + \vec{e}_j, \vec{\lambda}') \\&+ (-1)^{\mu(D)}\sum\limits_{j = 1}^{n} (-1)^{l(\lambda_1) + \dots + l(\lambda_{j - 1})} \sum\limits_{D = E * V_j \text{ (vertical)}} (-1)^{s(\lambda_j, \lambda_j')} \sum\limits_{\lambda_j' \in \text{UE}(\lambda_j)} (E, \vec{N} + \vec{e}_j, \vec{\lambda}') \\ \partial_3(D, \vec{N}, \vec{\lambda}) &= (-1)^{\mu(D)}\sum\limits_{j = 1}^{n} (-1)^{l(\lambda_1) + \dots + l(\lambda_{j - 1})} \sum\limits_{\lambda_j' \in \text{EC}(\lambda_j)} (-1)^{s(\lambda_j, \lambda_j')} (D, \vec{N}, \vec{\lambda}') \\ \partial_4(D, \vec{N}, \vec{\lambda}) &= (-1)^{\mu(D)}\sum\limits_{j = 1}^{n} (-1)^{l(\lambda_1) + \dots + l(\lambda_{j - 1})} \sum\limits_{\lambda_j' \in \text{IR}(\lambda_j)} (-1)^{s(\lambda_j, \lambda_j')} (D, \vec{N} - \lambda_{j, 1}\vec{e}_j, \vec{\lambda}') \\&+ (-1)^{\mu(D)}\sum\limits_{j = 1}^{n} (-1)^{l(\lambda_1) + \dots + l(\lambda_{j})} \sum\limits_{\lambda_j' \in \text{FR}(\lambda_j)} (-1)^{s(\lambda_j, \lambda_j')} (D, \vec{N} - \lambda_{j, m_j}\vec{e}_j, \vec{\lambda}')\end{align*}\label{def8}\end{definition}

\begin{remark}In the case that all $s_j = 0$, these signs agree with the signs of \cite[Definitions 4.1-4.3]{MS}, with the exception of the type II differential.\end{remark}

\begin{lemma}$(\CDP_*, \partial)$ is a chain complex.\label{lem6}\end{lemma}

\begin{proof}The proof is similar to that of \cite[Lemma 4.4]{MS}, which is the same case analysis of Lemma $\ref{lem5}$, except where we keep track of signs. In the case of
\begin{align*}(\partial_4)^2 + \partial_3 \partial_4 + \partial_4 \partial_3 = 0\end{align*}
we still have all but two cases cancelling in pairs by reversing the order of the two operations. These two cases are
\begin{itemize}

\item Two initial reductions and an elementary coarsening at the beginning, followed by an initial reduction. The former has sign
\begin{align*}\lambda_1 s_j + \lambda_2 s_j \pmod 2\end{align*}
and the latter has sign
\begin{align*}1 + (\lambda_1 + \lambda_2)s_j \pmod 2\end{align*}
which is the opposite sign.

\item Two final reductions and an elementary coarsening at the end, followed by a final reduction. Note that final reductions have an extra sign of $l(\lambda_j)$ compared to initial reductions, so that including this extra sign, the former has sign
\begin{align*}&l(\lambda_j) + (l(\lambda_j) - 1) + \lambda_m s_j + \lambda_{m - 1} s_j \pmod 2\end{align*}
and the latter has sign
\begin{align*}&(l(\lambda_j) - 1) + (l(\lambda_j) - 1) + (\lambda_{m - 1} + \lambda_m)s_j \pmod 2\end{align*}
which is the opposite sign.

\end{itemize}

Finally, although we still have $\partial_2 \partial_4 + \partial_4 \partial_2 = 0$, the change to the sign of the type II differential gives a new set of cancellations
\begin{align*}\partial_1^2 + \partial_2 \partial_4 + \partial_4 \partial_2 = 0\end{align*}
For this case, suppose $D = A * E = E * A$ is the domain where $A = R * S$ is an annulus.
\begin{itemize}

\item If $A$ is a vertical annulus $V_j$, then $s(R) + s(S) = 1$, so that removing $R$ then $S$ from the front has sign $1$, while the type II differential that produces a unit enlargement at the front of $\lambda_j$ followed by the initial reduction of $\lambda_j$ has sign $0$. Also, removing $S$ then $R$ from the back has sign $0$ (since the Maslov index of the domain decreases once), while the type II differential that produces a unit enlargement at the end of $\lambda_j$ followed by the final reduction of $\lambda_j$ has sign $l(\lambda_j) + 1 + l(\lambda_j) = 1 \pmod 2$.

\item If $A$ is a horizontal annulus $H_j$, then $s(R) + s(S) = 0$, so that removing $R$ then $S$ from the front has sign $0$, while the type II differential that produces a unit enlargement at the front of $\lambda_j$ followed by the initial reduction of $\lambda_j$ has sign $1$. Also, removing $S$ then $R$ from the back has sign $1$ (since the Maslov index of the domain decreases once), while the type II differential that produces a unit enlargement at the end of $\lambda_j$ followed by the final reduction of $\lambda_j$ has sign $l(\lambda_j) + l(\lambda_j) = 0 \pmod 2$.

\end{itemize}\end{proof}

The analogue of Proposition $\ref{prop4}$ also holds over $\Z$.

\begin{prop}There is a filtration on $\CDP_*$ such that the associated graded has homology $\Z^{2^n} \otimes \Z[U]$. In particular, $H_k(\CDP_*)$ has rank at most
\begin{align*}\sum\limits_{l = 0}^{\lfloor k/2 \rfloor} \binom{n}{k - 2l}\end{align*}\label{prop9}\end{prop}

\begin{proof}The proof is identical to the proof of Proposition $\ref{prop4}$.\end{proof}


\bibliography{1flow}
\bibliographystyle{amsalpha}

\end{document}